\documentclass[10.5pt,a4paper]{article}
\usepackage{amsmath,amssymb,amsthm,amsfonts,latexsym,graphicx,subfigure}
\usepackage{indentfirst}
\usepackage{fancyhdr}
\usepackage{bbm}
\usepackage{enumerate,tikz}
\usepackage{accents,cases}
\usepackage{multirow}
\usepackage{array}
\usepackage{float}
\newcommand{\PreserveBackslash}[1]{\let\temp=\\#1\let\\=\temp}
\newcolumntype{C}[1]{>{\PreserveBackslash\centering}p{#1}}
\newcolumntype{R}[1]{>{\PreserveBackslash\raggedleft}p{#1}}
\newcolumntype{L}[1]{>{\PreserveBackslash\raggedright}p{#1}}

\makeatletter
\def\wbar{\accentset{{\cc@style\underline{\mskip8mu}}}}

\makeatother

\numberwithin{equation}{section}

\theoremstyle{plain}
\newtheorem{theorem}{Theorem}
\newtheorem{defn}{Definition}
\newtheorem{lemma}{Lemma}
\newtheorem{remark}{Remark}
\newtheorem{cor}{Corollary}
\newtheorem{pro}{Proposition}

\newtheorem{class}{Class}

\newcommand{\res}{\ensuremath{\mathrm{Res_0}}}

\begin{document}
\bibliographystyle{unsrt}
\title{Non-crossing chords of a polygon with forbidden positions}
\author{Dongyi Wei\footnotemark[1],
\and Demin Zhang\footnotemark[2],
\and Dong Zhang\footnotemark[3]}
\renewcommand{\thefootnote}{\fnsymbol{footnote}}
\footnotetext[1]{School of Mathematical Sciences and BICMR, Peking University, Beijing 100871, P. R. China.\\
Email addresses:
{\tt jnwdyi@163.com} (Dongyi Wei).
}
\footnotetext[2]{Zhaipo middle school, Xinxiang 453700, Henan, P. R. China.
Email addresses: {\tt 13569444933@139.com} (Demin Zhang).
}
\footnotetext[3]{LMAM and School of Mathematical Sciences, Peking University, Beijing 100871, P. R. China.\\
Email addresses:
{\tt dongzhang@pku.edu.cn} \, or \, {\tt 13699289001@163.com} (Dong Zhang).
}
\date{}
\maketitle

\begin{abstract}
In this paper, we systematically study non-crossing chords of simple polygons in the plane.
We first introduce the reduced Euler characteristic of a family of line-segments, and subsequently investigate the structure of the diagonals and epigonals of a polygon. Interestingly enough, the reduced Euler characteristic of a subfamily of diagonals and epigonals characterizes the geometric convexity of polygons.  In particular, an alternative and complete answer is given for a problem proposed by G. C. Shephard. Meanwhile, we extend such research to non-crossing diagonals and epigonals with forbidden positions in some appropriate sense. We prove that the reduced Euler characteristic of diagonals with forbidden positions only depends on the information involving  convex partitions by those forbidden diagonals, and it determines the shapes of polygons in a surprising way. Incidentally,  some kinds of generalized Catalan's numbers  naturally arise.
\vspace{0.3cm}


\vspace{0.3cm}

\end{abstract}

\allowdisplaybreaks
\maketitle

\section{Introduction}\label{sec:intro}

A {\sl polygon} is a closed curve, composed of a finite sequence of straight line segments. These segments are called its edges, and the points where two edges meet are the polygon's vertices. For simplicity, we restrict ourselves to simple polygons (no self-intersecting) whose vertices are in general position (no three vertices are collinear).

Given a polygon $P$, a {\sl chord} is a segment whose endpoints are non-consecutive vertices of $P$. A chord is called a {\sl diagonal} (resp., {\sl epigonal}) if it lies in the interior (resp., exterior) of $P$.

Suppose $P$ has $n$ vertices, which we will symbolically denote by $|P|=n$, where $n\ge 4$. Let $d_1$ be the number of diagonals, $d_2$ be the number of non-crossing  pairs of diagonals, and, in general, $d_i$ be the number of sets of $i$ diagonals of the polygon which are pairwise non-crossing. Particularly, $d_{n-3}$ is the number of triangulations of $P$, and for any $n\ge 4$, there exist polygons satisfying $d_{n-3}=1$, such as polygons in Class \ref{ex:class2} (see Fig.~\ref{fig:class2} below). The number $e_i$ is defined in a similar manner for epigonals, $i=1,2,\ldots,n-3$.  By these definitions, $e_1$ stands for the number of epigonals of $P$, and thus $e_1>0$ represents the non-convexity of $P$.   Besides, we have $d_i=e_i=0$ if $i>n-3$, and we always set $d_0=e_0=1$.
\begin{class}
\label{ex:class2}\rm
This is the family of all non-convex polygons with only three (consecutive) angles which are less than $\pi$. Such polygonal region can be obtained by deleting a convex polygonal region from a triangular region (see Fig.~\ref{fig:class2} below).
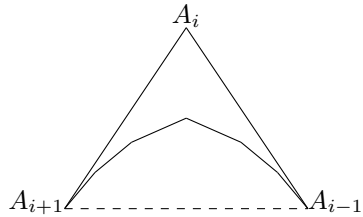
\begin{figure}[H]
\centering
\begin{tikzpicture}[scale=0.8]
\draw (0,0) to (-2,-3);
\draw (0,0) to (2,-3);
\draw (-2,-3) to (-1.5,-2.4);
\draw (2,-3) to (1.5,-2.4);
\draw (-0.9,-1.9) to (-1.5,-2.4);
\draw (0.9,-1.9) to (1.5,-2.4);
\draw (-0.9,-1.9) to (0,-1.5);
\draw (0.9,-1.9) to (0,-1.5);
\draw[dashed] (-2,-3) to (2,-3);
\node (Ai) at (0,0.2) {$A_i$};
\node (Ai+1) at (-2.46,-2.9) {$A_{i+1}$};
\node (Ai-1) at (2.46,-2.9) {$A_{i-1}$};
\end{tikzpicture}
\label{fig:class2}\caption{Illustration for polygons described in Class \ref{ex:class2}. In this polygon, only the  angles at the three vertices $A_{i-1}$, $A_i$ and $A_{i+1}$ are less than $\pi$.}
\end{figure}

\end{class}

\begin{defn}
Let $P$ be a simple polygon whose vertices are in general position. 
Now we define the {\sl reduced Euler characteristics} $\chi_d(P)=\sum_{i=0}^\infty (-1)^id_i$ and $\chi_e(P)=\sum_{i=0}^\infty (-1)^ie_i$.
\end{defn}
\begin{theorem}\label{th:main}
Let $P_n$ be a simple polygon with $n$ vertices  in general position. If $P_n$ is convex, then $\chi_d(P_n)=(-1)^{n+1}$ and $\chi_e(P_n)=1$. Otherwise, $\chi_d(P_n)=\chi_e(P_n)=0$.
\end{theorem}


It was already known that the complex of non-crossing diagonals of a convex polygon was spherical and so the convex case of Theorem \ref{th:main}  was true and a proof was published by Lee \cite{Lee}, as well as an alternative easier proof could be found in \cite{ZWZ}. Moreover, the first conclusion  $\chi_d(P_n)=(-1)^{n+1}$ in Theorem \ref{th:main}  is indeed the Euler-Poincar\'e formula for the associahedron, and as an extended version,  
we generalize this result to Theorem \ref{d_k}.

The non-convex case of Theorem \ref{th:main}  was proposed by Shephard \cite{Shephard}, and its first proof   was given by Braun and Ehrenborg \cite{BE2010}. In fact, they prove the simplicial complex of non-crossing diagonals in a polygon is a sphere or a ball of the expected dimension.
In Section \ref{sec:Euler-characteristic-nonconvex}, we give a new proof of  $\chi_d(P_n)=0$ and further prove $\chi_e(P_n)=0$ for non-convex case, and thus complete the proof of Theorem \ref{th:main}.

The main aim of this paper is to study non-crossing chords of simple polygons with restricted or  forbidden positions. Let $F$ be a set of finite points which are in general position in the plane, and let $M$ be a subset of  line-segments with end-points in $F$. The family of the sets of non-crossing segments in $M$ is denoted by
 $$NC[M]:=\{ J\subset M: \text{ the segments in }J\text{ are pairwise non-crossing}\}\cup\{\varnothing\},$$
 and the related counting numbers are $\nu_i(M):=\# \{ J\in NC[M]:\#J=i\}$, $i=0,1,\ldots$, where $\nu_0(M)=1$, and $\#$ is the counting function acting on finite sets. Denote by $\chi(M):=\sum_{i=0}^\infty (-1)^i\nu_i(M)$ the {\sl reduced Euler characteristic} of $M$. 

Now we concentrate on some polygons with restricted number of vertices, which can be viewed as a generalization of the convex case of Theorem \ref{th:main}.
\begin{defn}
Given $a\in \mathbb{N}^+$ and a polygon $P$ with $|P|=a(n+1)+2$ for some $n\in \mathbb{N}$, a diagonal of $P$ is said to be an {\sl $a$-diagonal} if there are $ka$ vertices between its two endpoints for some $k\in \mathbb{N}^+$. Let $M_d^a$ be the set of $a$-diagonals of $P$.
\end{defn}

\begin{theorem}\label{d_k}
Given $a\in \mathbb{N}^+ $ and $n\in \mathbb{N}^+$, let $P$ be a convex polygon with $a(n+1)+2$ vertices, and let $d_i(n,a)=\nu_i(M_d^a)$, $i=1,2,\ldots$. Then the reduced Euler characteristic $\chi(M_d^a)$ can be simplified as $(-1)^nd_n(n,a-1)$. Furthermore,  we have an inductive formula
$$d_k(n,a)=\frac{a(n+1)+2}{2k}\sum_{i_1+i_2=n-1}\sum_{j_1+j_2=k-1}d_{j_1}(i_1,a)d_{j_2}(i_2,a),$$
and then we obtain a closed formula $d_k(n,a)=\frac{1}{k+1}{a(n+1)+k+1\choose k}{n\choose k}$ for any $k\in \mathbb{N}^+$.
\end{theorem}

Note that every $J\in NC[M_d]$ provides a partition of $P$ by non-crossing diagonals. Given $M\subset M_d$, let
$$NC_c[M]=\{J\in NC[M]\colon\,J\text{ provides a convex partition of }P\},$$
and let $NC_{nc}[M]=NC[M]\setminus NC_c[M]$. It is noteworthy that the reduced Euler characteristic of a set of diagonals only depends on the corresponding convex partitions (see Theorem \ref{th:2} below). This plays a central role in the development of our ideas and results. 
\begin{theorem}\label{th:2}
Given $J\in NC[M_d]$, then there holds
\begin{equation}\label{eq:first-main}
\chi(M_d\setminus J)=(-1)^{|P|+1}\sum_{I\in NC_c[J]} (-1)^{\#I}.
\end{equation}
Moreover, we have the following conclusions:
\begin{enumerate}[(1)]
\item If $J\in NC_{nc}[M_d]$, then $\chi(M_d\setminus J)=0$.

\item Suppose $J\in NC_c[M_d]$, then the following statements hold.
 \begin{enumerate}[({2}a)]
\item 
If  $NC_c[J]$  has a unique minimal set $J'$, then $\chi(M_d\setminus J)=\begin{cases}0,&\text{ if }\,J'\ne J,\\
(-1)^{|P|+1+\# J},&\text{ if }\,J'=J.
\end{cases}$

\item If $J$ is not the union of all the minimal sets in $NC_c[J]$, then $\chi(M_d\setminus J)=0$.

\item If $J'\subset \bigcap_{I\in NC_c[J]}I$, then $\chi(M_d\setminus J)=\prod_{k=1}^m\chi(M_d(P^k)\setminus J)$, where $m=\#J'+1$, and $P^1,\ldots,P^m$ are the sub-polygons divided by $J'$.
\end{enumerate}
\end{enumerate}
\end{theorem}

Theorem \ref{th:2} could be used to determine the type of polygons with $\chi(M_d\setminus J)\ne 0$ for some fixed $J\in NC[M_d]$.
As an application, the following proposition indicates the fruitfulness of the topologies of simplicial complexes related to restricted diagonals of polygons. We construct a family of polygons to realize the proof.

\begin{pro}\label{pro:every}
 For every $l\in \mathbb{Z}$, there exists a polygon $P$ and $J\in NC[M_d]$ such that $\chi(M_d\setminus J)=l$.
\end{pro}

We provide Theorem \ref{th:main1} as a non-trivial application of Theorem \ref{th:2}  which also possesses independent interest in the study of typical polygons. 
First, we list a zoo of polygons which will be used in the next result.

\begin{class}\label{ex:class1}\rm
This is a special family of non-convex polygons with only one angle larger than $\pi$. For detailed descriptions, these polygons possess the properties that $\angle A_{i+1}A_iA_{i-1}>\pi>\angle A_{i+2}A_iA_{i-2}$, and $(\angle A_{i+2}A_i A_{i-1}-\pi)(\angle A_{i+1}A_iA_{i-2}-\pi)>0$ (see Figs.~\ref{fig:class1-1} and \ref{fig:class1-2}).

The class of such polygons is a special subclass of Class \ref{ex:class5} with a restriction that the unique special vertex $A_i$ lies in the region I or region III (see Fig.~\ref{fig:a}).
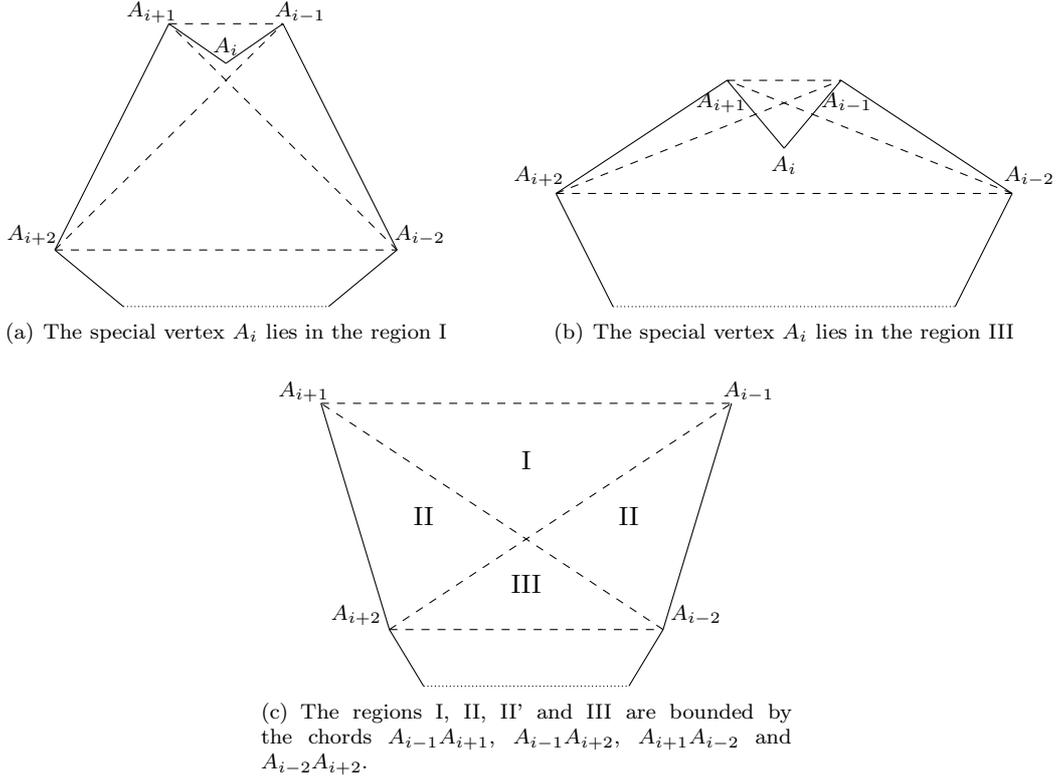
\begin{figure}[H]\centering
\subfigure[The special vertex $A_i$ lies in the region I]{
\begin{tikzpicture}[scale=1.5]
\draw (0,1.65) to (-0.5,2);
\draw (0,1.65) to (0.5,2);
\draw (-1.5,0) to (-0.5,2);
\draw (1.5,0) to (0.5,2);
\draw (-0.9,-0.5) to (-1.5,0);
\draw (0.9,-0.5) to (1.5,0);
\draw[densely dotted] (-0.9,-0.5) to (0.9,-0.5);
\draw[dashed] (1.5,0) to (-0.5,2);
\draw[dashed] (-1.5,0) to (0.5,2);
\draw[dashed] (-0.5,2) to (0.5,2);
\draw[dashed] (-1.5,0) to (1.5,0);
\node (Ai) at (0,1.8) {\footnotesize  $A_i$};
\node (Ai+1) at (-0.65,2.1) {\footnotesize $A_{i+1}$};
\node (Ai-1) at (0.65,2.1) {\footnotesize $A_{i-1}$};
\node (Ai+2) at (-1.7,0.13) {\footnotesize $A_{i+2}$};
\node (Ai-2) at (1.7,0.13) {\footnotesize $A_{i-2}$};
\end{tikzpicture}\label{fig:class1-1}
}
\;\;
\subfigure[The special vertex $A_i$ lies in the region III]{

\begin{tikzpicture}[scale=1.5]
\draw (0,0.4) to (-0.5,1);
\draw (0,0.4) to (0.5,1);
\draw (-2,0) to (-0.5,1);
\draw (2,0) to (0.5,1);
\draw (-1.5,-1) to (-2,0);
\draw (1.5,-1) to (2,0);
\draw[densely dotted] (-1.5,-1) to (1.5,-1);
\draw[dashed] (2,0) to (-0.5,1);
\draw[dashed] (-2,0) to (0.5,1);
\draw[dashed] (-0.5,1) to (0.5,1);
\draw[dashed] (-2,0) to (2,0);
\node (Ai) at (0,0.25) {\footnotesize  $A_i$};
\node (Ai+1) at (-0.55,0.8) {\footnotesize $A_{i+1}$};
\node (Ai-1) at (0.55,0.8) {\footnotesize $A_{i-1}$};
\node (Ai+2) at (-2.15,0.15) {\footnotesize $A_{i+2}$};
\node (Ai-2) at (2.15,0.15) {\footnotesize $A_{i-2}$};
\end{tikzpicture}\label{fig:class1-2}
}

\subfigure[The regions I, II, II' and III are bounded by the chords $A_{i-1}A_{i+1}$, $A_{i-1}A_{i+2}$, $A_{i+1}A_{i-2}$ and $A_{i-2}A_{i+2}$.]{
\begin{tikzpicture}[scale=1.5]
\draw (-1.2,0) to (-1.8,2);
\draw (1.2,0) to (1.8,2);
\draw (-0.9,-0.5) to (-1.2,0);
\draw (0.9,-0.5) to (1.2,0);
\draw[densely dotted] (-0.9,-0.5) to (0.9,-0.5);
\draw[dashed] (1.2,0) to (-1.8,2);
\draw[dashed] (-1.2,0) to (1.8,2);
\draw[dashed] (-1.8,2) to (1.8,2);
\draw[dashed] (-1.2,0) to (1.2,0);
\node (Ai+1) at (-1.95,2.1) {\footnotesize $A_{i+1}$};
\node (Ai-1) at (1.95,2.1) {\footnotesize $A_{i-1}$};
\node (Ai+2) at (-1.49,0.13) {\footnotesize $A_{i+2}$};
\node (Ai-2) at (1.49,0.13) {\footnotesize $A_{i-2}$};
\node (I) at (0,1.5) {  I};
\node (II1) at (-0.9,1) {  II};
\node (II2) at (0.9,1) {  II};
\node (III) at (0,0.4) {  III};
\end{tikzpicture}\label{fig:a}
}
\label{fig:class1}\caption{Illustration for polygons described in Class \ref{ex:class1}.}
\end{figure}

\end{class}

\begin{class}\label{ex:class3}\rm
This class of polygons are constructed in an elementary manner, where each polygonal region can be obtained by deleting a triangle region or a polygonal region in Class \ref{ex:class2} along an edge (or two neighbouring edges) of a convex polygonal region (see Fig.~\ref{fig:class3}).
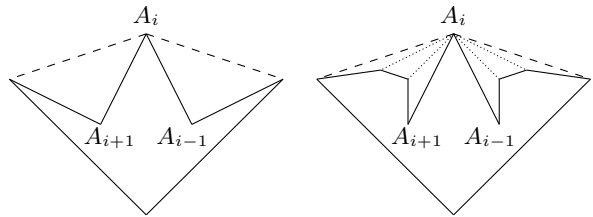
\begin{figure}[H]
\begin{center}
\begin{tikzpicture}[scale=1.2]
\draw (0,0) to (-0.5,-1);
\draw (0,0) to (0.5,-1);
\draw (-1.5,-0.5) to (-0.5,-1);
\draw (1.5,-0.5) to (0.5,-1);
\draw (-1.5,-0.5) to (0,-2);
\draw (1.5,-0.5) to (0,-2);
\draw[dashed] (0,0) to (-1.5,-0.5);
\draw[dashed] (0,0) to (1.5,-0.5);
\node (Ai) at (0,0.2) {\small $A_i$};
\node (Ai-1) at (0.4,-1.15) {\small $A_{i-1}$};
\node (Ai+1) at (-0.4,-1.15) {\small $A_{i+1}$};
\end{tikzpicture}
\;\;
\begin{tikzpicture}[scale=1.2]
\draw (0,0) to (-0.5,-1);
\draw (0,0) to (0.5,-1);
\draw (-0.5,-0.5) to (-0.5,-1);
\draw (0.5,-0.5) to (0.5,-1);
\draw (-0.5,-0.5) to (-0.8,-0.4);
\draw (0.5,-0.5) to (0.8,-0.4);
\draw (-1.5,-0.5) to  (-0.8,-0.4);
\draw (1.5,-0.5) to (0.8,-0.4);
\draw (-1.5,-0.5) to (0,-2);
\draw (1.5,-0.5) to (0,-2);
\draw[dashed] (0,0) to (-1.5,-0.5);
\draw[dashed] (0,0) to (1.5,-0.5);
\draw[densely dotted] (0,0) to (-0.5,-0.5);
\draw[densely dotted] (0,0) to (0.5,-0.5);
\draw[densely dotted] (0,0) to (-0.8,-0.4);
\draw[densely dotted] (0,0) to (0.8,-0.4);
\node (A) at (0,0.2) {\small $A_i$};
\node (Ai-1) at (0.4,-1.15) {\small $A_{i-1}$};
\node (Ai+1) at (-0.4,-1.15) {\small $A_{i+1}$};
\end{tikzpicture}
\end{center}
\label{fig:class3}\caption{Illustration for polygons described in Class \ref{ex:class3}. Such polygons satisfy $\angle A_{i-1}A_iA_{i+1}< \pi$.}
\end{figure}
\end{class}

\begin{class}\label{ex:class4}\rm
This class of polygons are constructed in an elementary manner, in which each  polygonal regions can be obtained by gluing a triangle region and a convex polygonal region along the edge $A_{i+1}A_{i-1}$ (see Fig.~\ref{fig:class4}).
\begin{figure}[H]
\centering
\begin{tikzpicture}[auto]
\draw (0,2) to (-0.5,1);
\draw (0,2) to (0.5,1);
\draw (-2,0) to (-0.5,1);
\draw (2,0) to (0.5,1);
\draw (0,-1) to (-2,0);
\draw (0,-1) to (2,0);
\draw[dashed] (0.5,1) to (-0.5,1);
\node (Ai) at (0,1.65) {\small $A_i$};
\node (Ai-1) at (0.85,1.1) {\small $A_{i-1}$};
\node (Ai+1) at (-0.85,1.1) {\small $A_{i+1}$};
\end{tikzpicture}
\label{fig:class4}\caption{Illustration for polygons described in Class \ref{ex:class4}. Such polygons satisfy $\angle A_{i-1}A_iA_{i+1}< \pi$.}
\end{figure}
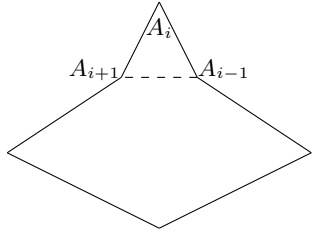
\end{class}

\begin{class}\label{ex:class5}\rm
This class of polygons are constructed in an elementary manner, in which each polygonal regions can be obtained by deleting a triangle region from a convex polygonal region along the edge $A_{i+1}A_{i-1}$ (see Fig.~\ref{fig:class5}).
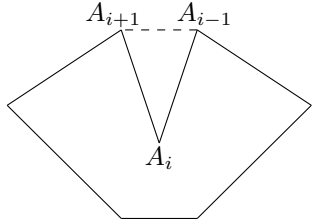
\begin{figure}[H]
\centering
\begin{tikzpicture}[auto]
\draw (0,-0.5) to (-0.5,1);
\draw (0,-0.5) to (0.5,1);
\draw (-2,0) to (-0.5,1);
\draw (2,0) to (0.5,1);
\draw (-0.5,-1.5) to (-2,0);
\draw (0.5,-1.5) to (2,0);
\draw (-0.5,-1.5) to(0.5,-1.5);
\draw[dashed] (0.5,1) to (-0.5,1);
\node (A) at (0,-0.7) {$A_i$};
\node (Ai-1) at (0.6,1.2) {$A_{i-1}$};
\node (Ai+1) at (-0.6,1.2) {$A_{i+1}$};
\end{tikzpicture}
\label{fig:class5}\caption{Illustration for polygons described in Class \ref{ex:class5}.  Such polygons satisfy $\angle A_{i-1}A_iA_{i+1}> \pi$.}
\end{figure}
\end{class}

\begin{class}\label{ex:class6}\rm
This family of polygonal regions can be obtained by gluing one (or two) polygonal region in Class \ref{ex:class2} and a polygonal region in Class \ref{ex:class1} (see Fig.~\ref{fig:class6}). 
\begin{figure}[H]
\centering
\begin{tikzpicture}[scale=1.5]
\draw (0,0) to (-0.5,-1);
\draw (0,0) to (0.5,-1);
\draw (-0.5,-0.5) to (-0.8,-0.4);
\draw (0.5,-0.5) to (0.8,-0.4);
\draw (-1.5,-0.5) to  (-0.8,-0.4);
\draw (1.5,-0.5) to (0.8,-0.4);
\draw (-1.5,-0.5) to (-0.9,1.2);
\draw (1.5,-0.5) to ( 0.9,1.2);
\draw (-0.9,1.2) to (0,1.2);
\draw ( 0.9,1.2) to (0,1.2);
\draw[densely dotted] (0,0) to (-0.5,-0.5);
\draw[densely dotted] (0,0) to (0.5,-0.5);
\draw (-0.4,-0.64) to (-0.5,-0.5);
\draw (0.4,-0.64) to (0.5,-0.5);
\draw (-0.4,-0.64) to (-0.5,-1);
\draw (0.4,-0.64) to (0.5,-1);
\draw[densely dotted] (-0.4,-0.64) to (0,0);
\draw[densely dotted] (0.4,-0.64) to (0,0);
\draw[dashed] (0,0) to (-0.8,-0.4);
\draw[dashed] (0,0) to (0.8,-0.4);
\node (A) at (0,0.2) {$A_i$};
\node (Ai-1) at (-0.4,-1.15) {\small $A_{i-1}$};
\node (Ai+1) at (0.4,-1.15) {\small $A_{i+1}$};
\end{tikzpicture}
\label{fig:class6}\caption{Illustration for polygons described in Class \ref{ex:class6}.  Such polygons satisfy $\angle A_{i-1}A_iA_{i+1}> \pi$.}
\end{figure}
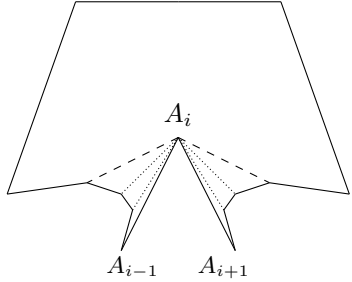
\end{class}

\begin{theorem}\label{th:main1}
Linearly order the vertices of a polygon $P$ in counter-clockwise direction, $A_1,A_2,\ldots,A_n$, where $n:=|P|\ge 5$. Given   $i\in\{1,2,\ldots,n\}$, we have the following statements. Here   all indices are specified mod $n$.
\begin{enumerate}[(A)]
\item $\chi(M_d\setminus\{A_iA_j:j\ne i-1,i,i+1\})\ne 0$ \; $\Leftrightarrow$ \; $P$ belongs to Class \ref{ex:class1} or Class \ref{ex:class2} or Class \ref{ex:class6} (see Figs.~\ref{fig:class1}, \ref{fig:class2} and \ref{fig:class6}).

\item $\chi(M_e\setminus\{A_iA_j:j\ne i-1,i,i+1\})\ne 0$ \; $\Leftrightarrow$ \;  $P$ is convex or belongs to Class \ref{ex:class3}  (see Fig.~\ref{fig:class3}).

\item $\chi(M_d\setminus\{ A_{i-1}A_{i+1}\})\ne 0$ \; $\Leftrightarrow$ \;  $P$ belongs to Class \ref{ex:class4} (see Fig.~\ref{fig:class4}).

\item $\chi(M_e\setminus\{ A_{i-1}A_{i+1}\})\ne 0$ \; $\Leftrightarrow$ \; $P$ is convex or belongs to
Class \ref{ex:class2} or Class \ref{ex:class5}  (see Figs.~\ref{fig:class2} and \ref{fig:class5}).
\end{enumerate}
\end{theorem}

The paper is organized as follows. In Section \ref{sec:intro}, we give the background, introduction, preliminary, and show the main theorems of this paper. A brief discussion of reduced Euler characteristic and a proof of Theorem \ref{th:main} (i.e., Shephard's problem) is in Section \ref{sec:Euler-characteristic-nonconvex}. Auxiliary results on reduced Euler characteristic  and the proof of Theorem \ref{th:2} and Proposition \ref{pro:every} are proposed in Section \ref{sec:main-proof}. Detailed proofs of Theorem \ref{th:main1} and Theorem \ref{d_k} are presented respectively in Section \ref{sec:thm4} and Section \ref{sec:furtherresults} with further results. Additional illustrations with a few remarks are provided in the appendix.

\section{Reduced Euler characteristic for family of segments and the proof of Theorem \ref{th:main}}
\label{sec:Euler-characteristic-nonconvex}

First we list some basic and elementary facts which will be  used in the sequel. The proofs are very basic and we put them in the Appendix for reader's convenience.

\begin{pro}\label{pro:d1>=1}
If $n\ge4$, then $P_n$ has diagonals, i.e., $d_1\ge 1$.
\end{pro}

\begin{pro}\label{pro:triangulation}
For a set of non-crossing diagonals, $J\in NC[M_d]$, there exists $J'\supset J$ which divides $P_n$ into triangles. Particularly, for any $n\ge 3$, $d_{n-3}\ge 1$.
\end{pro}

\begin{remark}
If $P_{n+3}$ is convex, then $d_n$ is known as the Catalan number. It is well-known that  $d_n=\frac{1}{n+1}{2n+2 \choose n}$.
\end{remark}

\subsection{reduced Euler characteristic of a set of segments in the plane}

Let $M$ be a set of segments in the plane. For $A\subset M$, let
 $$NC[A]=\{ J\subset A: \text{ the segments in }J\text{ are pairwise non-crossing}\}\cup\{\varnothing\}$$
 and let $\nu_i(A)=\# \{ J\in NC[A]:\#J=i\}$, $i=0,1,\cdots$. Here we set $\nu_0(A)=1$. Denote by $\chi(A):=\sum_{i=0}^\infty (-1)^i\nu_i(A)$ the reduced Euler characteristic of $A$.

\begin{remark}
(1) $\chi(\varnothing)=1$, $\chi(\{v\})=0$ for any $v\in A$.

(2) If $\nu_i(A)=0$, then $\nu_{i+1}(A)=0$.

(3) If $\# A=n$ and $i>n$, then $\nu_i(A)=0$. So $\chi(A)=\sum_{i=0}^n(-1)^i \nu_i(A)$ is a finite sum  and thus it is well-defined.
\end{remark}

\begin{pro}\label{pro:chi-induction}
If $v\in A$, then $\chi(A)=\chi(A\setminus \{v\})-\chi(A_v)$, where $A_v$ collects the segments in $A\setminus \{v\}$ which are non-crossing with $v$.
\end{pro}

Proposition \ref{pro:chi-induction} is a general fact about the reduced Euler characteristic of flag simplicial complexes, connected the reduced Euler characteristic of the complex with that of the deletion and the link of a vertex. For reader's convenience, we give a proof in the Appendix.

\begin{defn}
Let $H\in NC[A]$. We call $H$ a {\sl center} of $A$, if for any $J\in NC[A]$, there exists $J'\in NC[A]$ such that $J'\supset J$ and $J'\cap H\ne\varnothing$. If $A$ has a center, then we call it a {\sl star} set.
\end{defn}

\begin{pro}\label{pro:starset=0}
If $A$ is a star set, then $\chi(A)=0$.
\end{pro}

\begin{proof}
We do induction on $\#A$. If $\#A=1$, then $\chi(A)=\nu_0(A)-\nu_1(A)=1-1=0$. Suppose that for any star set $A$ with $\# A<n$, $\chi(A)=0$, then for any star set $A$ with $\#A=n$, we shall prove that $\chi(A)$ still equals to $0$.

Let $H$ be a center of $A$. Thus, $H\ne\varnothing$. If $A=H$, then $\chi(A)=\sum_{i\ge 0}(-1)^i{\# A\choose i}=(-1+1)^{\# H}=0$. Otherwise, let $v\in A\setminus H$. Then   Proposition \ref{pro:chi-induction} implies that $\chi(A)=\chi(A\setminus \{v\})-\chi(A_v)$. Obviously, $\# A_v\le \#(A\setminus \{v\})=\# A-1=n-1$.

For any $J\in NC[A\setminus\{v\}]$, we have $J\in NC[A]$ and thus there exists $J'\in NC[A]$ with $J'\supset J$ such that $J'\cap H\ne\varnothing$. Hence, $J'\setminus\{v\}\supset J$ and $(J'\setminus\{v\})\cap H=J'\cap (H\setminus \{v\})=J'\cap H\ne\varnothing$. Therefore, $H$ is a center  of $A\setminus \{v\}$, which means that $A\setminus \{v\}$ is a star set.

For any $J\in NC[A_v]$, we have $J\cup \{v\}\in NC[A]$ and thus there exists $u\in H$ such that $J\cup \{v\}\cup\{u\}\in NC[A]$ and thus $u\in A_v$. Let $J'=J\cup \{v\}\cup\{u\}$. Then  $J'\cap A_v\in NC[A_v]$, $J'\cap A_v\supset J$, and $\varnothing \ne (J'\cap A_v)\cap H\cap A_v\ni u$. Therefore, $H\cap A_v$ is a center of $A_v$, and hence $A_v$ is a star set.

By the hypothesis of induction, we have $\chi(A\setminus \{v\})=0$ and $\chi(A_v)=0$. Therefore, $\chi(A)=0$.
\end{proof}

\subsection{A solution of Shephard's problem (i.e., non-convex case of Theorem \ref{th:main}) }

\begin{proof}[Proof of Theorem \ref{th:main} for non-convex case]
Since $P$ is non-convex, it has more than three vertices. We assume $|P|\ge4$ and $\angle A_1 > \pi$.
Let $H$ be a set of diagonals with an end-point $A_1$ (see Fig.~\ref{fig:Shephard-proof}). Then $H\in NC[M_d]$.

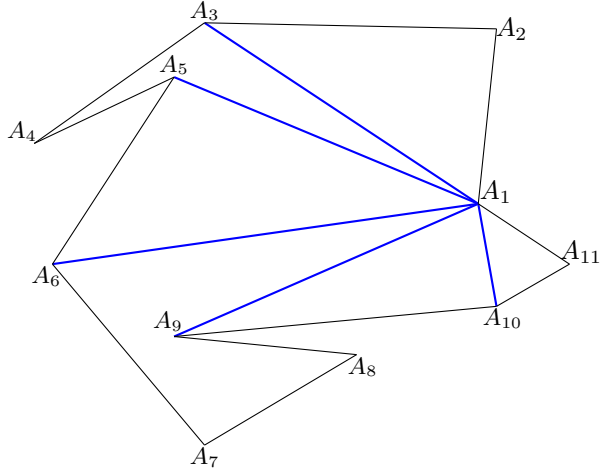
\begin{figure}[H]
\centering
\begin{tikzpicture}[scale=0.8]
\draw (0,0)[thick,blue] to (0.3,-1.7);
\draw (0,0)[thick,blue] to (-5,-2.2);
\draw (0,0)[thick,blue] to (-7,-1);
\draw (0,0)[thick,blue] to (-5,2.1);
\draw (0,0)[thick,blue] to (-4.5,3);
\draw (0,0) to (0.3,2.9);
\draw (0,0) to (1.5,-1);
\draw (-7.3,1) to (-4.5,3);
\draw (-4.5,-4) to (-7,-1);
\draw (-2,-2.5) to (-5,-2.2);
\draw (-7.3,1) to (-5,2.1);
\draw (-2,-2.5) to (-4.5,-4);
\draw (1.5,-1) to (0.3,-1.7);
\draw (-4.5,3) to (0.3,2.9);
\draw (-5,2.1) to (-7,-1);
\draw (0.3,-1.7) to (-5,-2.2);
\node (A) at (0.29,0.2) {$A_1$};
\node (T) at (1.69,-0.8) {$A_{11}$};
\node (S) at (0.55,2.9) {$A_2$};
\node (P8) at (0.4,-1.9) {\small $A_{10}$};
\node (P7) at (-5.1,-1.95) {\small$A_9$};
\node (P6) at (-1.9,-2.7) {\small$A_8$};
\node (P5) at (-4.5,-4.2) {\small$A_7$};
\node (P4) at (-7.1,-1.2) {\small$A_6$};
\node (P3) at (-5,2.3) {\small$A_5$};
\node (P2) at (-7.5,1.2) {\small$A_4$};
\node (P1) at (-4.5,3.2) {\small$A_3$};
\end{tikzpicture}
\caption{\label{fig:Shephard-proof} Illustration for the proof of Theorem \ref{th:main} (B). In this polygon, we can take $H=\{A_1A_3,A_1A_5,A_1A_6,A_1A_9,A_1A_{10}\}$.}
\end{figure}

For any $J\in NC[M_d]$, by Proposition \ref{pro:triangulation}, there exists $J'\in NC[M_d]$ such that $J'\supset J$ and $J'$ divides $P$ into triangles. Since the angle $A_1$ can not be an angle of a triangle, there is someone (a diagonal) in $J'$ such that the vertex  $A_1$ is its end-point. Therefore $H\cap J'\ne \varnothing$.
So $M_d$ is a star set, and then by Proposition \ref{pro:starset=0}, we get $\chi(M_d)=0$.

For the case of $M_e$, note that there exists an epigonal as a side of the convex hull of $P$. Such epigonal must be non-crossing with other epigonals. This means that such epigonal is a center of $M_e$. Consequently, $M_e$ is a star set, and by Proposition \ref{pro:starset=0}, we get $\chi(M_e)=0$. Combining with the convex case of Theorem \ref{th:main}, we complete the proof.
\end{proof}

Finally, we show a generalization of the non-convex case in Theorem \ref{th:main}. 

\begin{pro}\label{pro:finite-points}
Consider the set $F$ of finite points in the plane, and the set $S_2(F)$ of all the line-segments whose end-points lie in $F$. Let $S\subset S_2(F)$. If $S$ contains an edge of the convex polygon $Pconv(F)$, then $\chi(S)=0$. Here, $Pconv(F)$ is the boundary polygon of the convex hull $conv(F)$.
\end{pro}

Let $F$ be the collections of vertices of a non-convex polygon $P$, and let $S=M_e(P)$. Then Proposition \ref{pro:finite-points} immediately implies $\chi(M_e)=0$.

\section{Auxiliary results and the proof of Theorem \ref{th:2} and Proposition \ref{pro:every}}
\label{sec:main-proof}

\begin{lemma}\label{lem:1}
Let $J\subset M_d$ be a set of pairwise non-crossing diagonals. For $I\subset J$, $I$ divides $P$ into $1+\# I$ sub-polygons, denoted by $P_{I,k}$, $k=1,2,\ldots,\# I+1$. Then $\chi(M_d\setminus J)=\sum_{I\subset J} \prod_{k=1}^{\#I+1} \chi(M_d(P_{I,k}))$, where $M_d(P_{I,k})$ is the set of diagonals of $P_{I,k}$.
\end{lemma}

\begin{proof}
We classify the sets in $NC[M_d]$ via their intersections  with $J$. It follows from the principle of inclusion-exclusion that
\begin{align*}
\nu_j(M_d\setminus J)=& \sum_{S\in NC[M_d],\# S=j,S\cap J=\varnothing} 1
\\=& \sum_{S\in NC[M_d],\# S=j} 1+ \sum_{I\subset J,1\le \# I\le j}(-1)^{\# I} \sum_{S\in NC[M_d],\# S=j,S\cap J\supset I} 1
\\=& \nu_j(M_d)+ \sum_{I\subset J,1\le \# I\le j}(-1)^{\# I} \nu_{j-\# I}(\widehat{M_d\setminus I})
\\=& \sum_{I\subset J,\# I\le j}(-1)^{\# I}\nu_{j-\# I}(\widehat{M_d\setminus I}),
\end{align*}
where $\widehat{M_d\setminus I}$ denotes the set of diagonals which are non-crossing with the diagonals in $I$. Then, according to the definition of reduced Euler characteristic and the above equality, we have
\begin{align*}
\chi(M_d\setminus J)&=\sum_{j=0}^\infty (-1)^j \nu_j(M_d\setminus J)
\\&= \sum_{j=0}^\infty (-1)^j \sum_{I\subset J,\# I\le j}(-1)^{\# I}\nu_{j-\# I}(\widehat{M_d\setminus I})
\\&=\sum_{I\subset J}  \sum_{j=\# I}^\infty (-1)^{j-\# I} \nu_{j-\# I}(\widehat{M_d\setminus I})
\\&=\sum_{I\subset J}  \chi(\widehat{M_d\setminus I})
\\&=\sum_{I\subset J} \prod_{k=1}^{\#I+1} \chi(M_d(P_{I,k})).
\end{align*}
The last equality is a direct consequence of the product formula of reduced Euler characteristic.
\end{proof}


A direct calculation following Lemma \ref{lem:1} gives
\begin{align*}
\chi(M_d\setminus J)&=\sum_{I\subset J} \prod_{k=1}^{\#I+1} \chi(M_d(P_{I,k}))
\\&=\sum_{I\subset J,~P_{I,k}\text{ convex},\forall k} \prod_{k=1}^{\#I+1} (-1)^{|P_{I,k}| +1}
\\&=\sum_{I\subset J,~P_{I,k}\text{ convex},\forall k} (-1)^{\sum_{k=1}^{\#I+1}(|P_{I,k}| +1)}
\\&=\sum_{I\in NC_c[J]} (-1)^{|P|+2\#I +\#I+1}
\\&=(-1)^{|P|+1}\sum_{I\in NC_c[J]} (-1)^{\#I}.
\end{align*}

So, we complete the proof of \eqref{eq:first-main}, which is the main part of Theorem \ref{th:2}. Next, we focus on the other parts.

\begin{cor}\label{cor:convex=0}
If $P$ is convex, and $J\in NC[M_d]\setminus \{\varnothing\}$, then $\chi(M_d\setminus J)=0$.
\end{cor}
\begin{proof}
Note that $\sum_{I\in NC_c[J]}(-1)^{\#I}=\sum_{I\subset J}(-1)^{\#I}=(-1+1)^{\# J}=0$.
\end{proof}

The following Lemma \ref{lem:d2} is another form of \eqref{eq:first-main} in Theorem \ref{th:2}.

\begin{lemma}\label{lem:d2}
Let $J\subset M_d$ be a nonempty subset of pairwise non-crossing diagonals. Then $\chi(M_d\setminus J)=(-1)^{|P|}\sum_{I\in NC_{nc}[J]} (-1)^{\#I}$.
\end{lemma}

\begin{proof}
Note that $\sum_{I\in NC[J]}(-1)^{\#I}=\sum_{I\subset J}(-1)^{\#I}=(-1+1)^{\# J}=0$. Thus, by Theorem \ref{th:2}, we have
\begin{align*}
\chi(M_d\setminus J)&=(-1)^{|P|+1}\sum_{I\in NC_c[J]} (-1)^{\#I}
\\&=(-1)^{|P|+1}\left(\sum_{I\in NC[J]}(-1)^{\#I}- \sum_{I\in NC_{nc}[J]} (-1)^{\#I}\right)
\\&=(-1)^{|P|}\sum_{I\in NC_{nc}[J]} (-1)^{\#I}.
\end{align*}
\end{proof}

\begin{pro}\label{pro:1}
If $J$ divides $P$ into sub-polygons containing non-convex one, then $\chi(M_d\setminus J)=0$.
\end{pro}

\begin{proof}
The case of $J=\varnothing$ reduces to Theorem \ref{th:main} (B). We suppose that $J\ne\varnothing$. Since $J\in NC_{nc}[M_d]$, it is easy to check that $NC_{nc}[J]=NC[J]$. Thus, combining with Lemma \ref{lem:d2}, we immediately obtain
\begin{align*}
\chi(M_d\setminus J)&=(-1)^{|P|}\sum_{I\in NC_{nc}[J]} (-1)^{\#I}
\\&=(-1)^{|P|}\sum_{I\subset J} (-1)^{\#I}
\\&=(-1)^{|P|}(-1+1)^{\#J}=0.
\end{align*}
\end{proof}

By Proposition \ref{pro:1}, we deduce Theorem  \ref{th:2} (1).

\begin{pro}\label{pro:2}
Suppose $J$ divides $P$ into convex polygons. Assume that there exists the unique minimal subset $J_c\subset J$ such that $P$ can be divided by $J_c$ into convex sub-polygons. Then $\chi(M_d\setminus J)=0$ if and only if $J_c\ne J$. Besides, if $J$ provides a minimal convex partition by non-crossing diagonals, then $\chi(M_d\setminus J)=(-1)^{|P|+\#J+1}$.
\end{pro}

\begin{proof}
Since $J_c$ is the unique minimal subset of $J$ which divides $P$ into convex polygons, for $I\subset J$, $I$ divides $P$ into convex polygons if and only if $J_c\subset I$. Combining with Lemma \ref{lem:d2}, we immediately obtain
\begin{align*}
\chi(M_d\setminus J)&=(-1)^{|P|+1}\sum_{I\in NC_c[M_d]} (-1)^{\#I}
\\&=(-1)^{|P|+1}\sum_{J_c\subset I\subset J} (-1)^{\#I}
\\&=(-1)^{|P|+1+\# J_c}\sum_{ I'\subset J\setminus J_c} (-1)^{\#I'}
\\&=(-1)^{|P|+1+\# J_c}\begin{cases}(-1+1)^{\# (J\setminus J_c)},&\text{ if }\;J\setminus J_c\ne\varnothing,\\
1,&\text{ if }\;J\setminus J_c=\varnothing,
\end{cases}
\\&=\begin{cases}0,&\text{ if }\;J_c\ne J,\\
(-1)^{|P|+1+\# J},&\text{ if }\;J_c=J.
\end{cases}
\end{align*}
\end{proof}

By Proposition \ref{pro:2}, we get Theorem  \ref{th:2} (2a).

\begin{pro}\label{pro:J'J}
Let $J\in NC_c[M_d]$ and let $J'\subset J$ satisfy $J'\subset I$, $\forall I\in NC_c[J]$. Then $\chi(M_d(P)\setminus J)=\prod_{k=1}^m\chi(M_d(P_k)\setminus J_k)$, where $m=\#J'+1$, and $P_1,\ldots,P_m$ are the sub-polygons divided by $J'$ and $J_k=(J\setminus J')\cap  M_d(P_k)$, $k=1,\ldots,m$.
\end{pro}

\begin{proof}
Let $NC_c[J_k,M_d(P_k)]=\{I\subset J_k:I\text{ divides }P_k \text{ into convex polygons}\}$. According to Theorem \ref{th:2}, we obtain
\begin{align*}
\chi(M_d\setminus J)&=(-1)^{|P|+1}\sum_{I\in NC_c[J]} (-1)^{\#I}
\\&=(-1)^{|P|+1}\sum_{I_k \in NC_c[J_k,M_d(P_k)],k=1,\ldots,m} (-1)^{\#J'+\sum_{k=1}^m\#I_k}
\\&=(-1)^{|P|+1+\#J'}\prod_{k=1}^m\sum_{I_k \in NC_c[J_k,M_d(P_k)]}(-1)^{\#I_k}
\\&=(-1)^{|P|+1+\#J'-\sum_{k=1}^m(|P_k|+1)}\prod_{k=1}^m(-1)^{|P_k|+1}\sum_{I_k \in NC_c[J_k,M_d(P_k)]}(-1)^{\#I_k}
\\&=\prod_{k=1}^m\chi(M_d(P_k)\setminus J_k).
\end{align*}
\end{proof}
By Proposition \ref{pro:J'J}, Theorem \ref{th:2} (2c) is proved. 
Using similar techniques, we can prove
\begin{lemma}\label{lem:e}
Assume that $P$ and its convex hull exactly bound $m$ polygons, 
$P^1,\ldots,P^m$. Let $J\subset M_e$ be a subset of pairwise non-crossing epigonals.  Then $\chi(M_e\setminus J)= \prod_{k=1}^{m} \chi(M_d(P^k)\setminus J)$.
\end{lemma}


\begin{defn}\label{def:xi}
Given a non-convex polygon $P$, $J\in NC_c[M_d]$ and $I\subset J$, let
$$\xi(I)=\begin{cases}
0,& \text{ if } I\in NC_{nc}[J]\setminus \{\varnothing\} \text{ or } I\in NC_c[J]\setminus \{J\},\\
1,& \text{ if } I=\varnothing \text{ or } I=J.
\end{cases}$$
\end{defn}

\begin{pro}\label{pro:xi-eta}
Let $P$ be a non-convex polygon and $J\in NC_c[M_d]$. Suppose $J_1,\ldots,J_m\in NC_c[J]$ are all the minimal sets. Then
$$\chi(M_d\setminus J)=(-1)^{|P|+\#J}\sum_{k=1}^m (-1)^{k}\sum_{1\le i_1<\cdots<i_k\le m}\xi(J_{i_1}\cup\cdots\cup J_{i_k}).$$
Assume $I_1,\ldots,I_m\in NC_{nc}[J]$ are all the maximal sets. Then
$$\chi(M_d\setminus J)=(-1)^{|P|}\sum_{k=1}^m (-1)^{k-1}\sum_{1\le i_1<\cdots<i_k\le m}\xi(I_{i_1}\cap\cdots\cap I_{i_k}).$$
\end{pro}

\begin{proof}

Given $I\in NC_c[J]$, let $\eta(I)=\sum_{I\subset I'\subset J}(-1)^{\#I'}$. Then $\eta(I)=\begin{cases}0,& I\ne J,\\(-1)^{\# J},&I=J,\end{cases}=(-1)^{\# J}\xi(I)$.

Let $J_1,\ldots,J_m\in NC_c[J]$ be all the minimal sets, i.e., for any $I\in NC_c[J]$, there exists $i\in \{1,\ldots,m\}$ such that $I\supset J_i$. It follows from Theorem \ref{th:2} and the principle of inclusion-exclusion that
\begin{align*}
\chi(M_d\setminus J)&=(-1)^{|P|+1}\sum_{I\in NC_c[J]} (-1)^{\#I}
\\&=(-1)^{|P|+1}\sum_{k=1}^m (-1)^{k-1}\sum_{1\le i_1<\cdots<i_k\le m}\eta(J_{i_1}\cup\cdots\cup J_{i_k})
\\&=(-1)^{|P|+\#J}\sum_{k=1}^m (-1)^{k}\sum_{1\le i_1<\cdots<i_k\le m}\xi(J_{i_1}\cup\cdots\cup J_{i_k}).
\end{align*}

Given $I\in NC_{nc}[J]$, then we have $\sum_{I'\subset I}(-1)^{\#I'}=\begin{cases}0,& I\ne \varnothing,\\1,&I=\varnothing,\end{cases}=\xi(I)$.

Let $I_1,\ldots,I_m\in NC_{nc}[J]$ be all the maximal sets, i.e., for any $I\in NC_{nc}[J]$, there exists $i\in \{1,\ldots,m\}$ such that $I\subset I_i$. Then Lemma \ref{lem:d2} together with the principle of inclusion-exclusion deduce that
\begin{align*}
\chi(M_d\setminus J)&=(-1)^{|P|}\sum_{I\in NC_{nc}[J]} (-1)^{\#I}
\\&=(-1)^{|P|}\sum_{k=1}^m (-1)^{k-1}\sum_{1\le i_1<\cdots<i_k\le m}\xi(I_{i_1}\cap\cdots\cap I_{i_k}).
\end{align*}
\end{proof}

\begin{remark}
In Proposition \ref{pro:xi-eta}, the family of the sets $J_1,\ldots,J_m$ (resp., $I_1,\ldots,I_m$) forms a Sperner family, i.e., none of the sets is contained in another.
\end{remark}

\begin{cor}\label{cor:union-minimal}
Let $P$ be a non-convex polygon and $J\in NC_c[M_d]$. Let $J_1,\ldots,J_m\in NC_c[J]$ be all the minimal sets. If $J_1\cup\cdots\cup J_m\ne J$, then $\chi(M_d\setminus J)=0$.
\end{cor}

\begin{proof}
Since $J_1\cup\cdots\cup J_m\ne J$, for any $1\le i_1<\cdots<i_k\le m$, $J_{i_1}\cup\cdots\cup J_{i_k}\ne J$. Thus by Definition \ref{def:xi}, $\xi(J_{i_1}\cup\cdots\cup J_{i_k})=0$, and Proposition \ref{pro:xi-eta} then implies $\chi(M_d\setminus J)=0$.
\end{proof}

According to Corollary \ref{cor:union-minimal}, we derive Theorem \ref{th:2} (2b).
The following result is an analogue of Corollary \ref{cor:union-minimal}.

\begin{cor}\label{cor:intersection-maximal}
Let $P$ be a non-convex polygon and $J\in NC_c[M_d]$. Let $I_1,\ldots,I_m\in NC_{nc}[J]$ be all the maximal sets. If $I_1\cap\cdots\cap I_m\ne \varnothing$, then $\chi(M_d\setminus J)=0$.
\end{cor}

\subsection{  Proof of Proposition \ref{pro:every} }
Now we show a proof of Proposition \ref{pro:every}. Note that Theorem \ref{th:2} (1) and (2) provide  the examples of the case $l\in\{-1,0,1\}$ of Proposition \ref{pro:every}  (for $l=0$, we can take $P$ non-convex and $J=\varnothing$, and for $l=\pm 1$, we can take $P$ convex and $J=\varnothing$). Therefore we only need to consider the case of $|l|>1$. We first pay attention to the case of $l>1$.

\begin{figure}[H]
\centering
\begin{tikzpicture}[scale=2]
\coordinate (B1) at (0,0);
\coordinate (C1) at (0.5,-0.866);
\coordinate (D1) at (1.366,-1.366);

\coordinate (B2) at (1,0);
\coordinate (C2) at (1,1);
\coordinate (D2) at (1+0.5,1.866);
\coordinate (e1) at (1,0.5); \draw  (e1) node [ left]{\small$e_1$};
\coordinate (e2) at (1/2+1/2+0.866/2,1/2+0.5/2); \draw  (e2) node [ above ]{\small$e_2$};
\coordinate (e3) at (1/2+1/2+0.866/2,0.5/2); \draw  (e3) node [ below ]{\small$e_3$};
\coordinate (C3) at (1+0.866+0.5,0.5-0.866);
\coordinate (B3) at (1+0.866,0.5);
\coordinate (D3) at (1+0.866+1.366,0.5-1.366);

\coordinate (B4) at (1+0.866+1,0.5);
\coordinate (C4) at (1+0.866+1,0.5+1);
\coordinate (D4) at (1+0.866+1+0.5,0.5+1.866);

\coordinate (B5) at (1+0.866+1+0.866,0.5+0.5);
\coordinate (C5) at (1+0.866+1+0.866+0.5,0.5+0.5-0.866);
\coordinate (D5) at (1+0.866+1+0.866+1.366,0.5+0.5-1.366);

\coordinate (B6) at (1+0.866+1+0.866+1,0.5+0.5);
\coordinate (C6) at (1+0.866+1+0.866+1,0.5+0.5+1);
\coordinate (D6) at (1+0.866+1+0.866+1+0.5,0.5+0.5+1.866);

\coordinate (B7) at (1+0.866+1+0.866+1+0.866,0.5+0.5+0.5);
\coordinate (B8) at (1+0.866+1+0.866+1+0.866+1,0.5+0.5+0.5);
\coordinate (B9) at (1+0.866+1+0.866+1+0.866+1+0.866,0.5+0.5+0.5+0.5);

\coordinate (C) at (2*\pgfmathresult, 1);

\draw (B1)node [above]{\scriptsize $A_1$};
\draw (B2)node [below]{\scriptsize$A_2$} ;
\draw (B3)node [above right]{\scriptsize$A_4$};
\draw  (C2) node [above left]{\scriptsize$A_6$};
\draw(D2) node [ left]{\scriptsize$A_5$};
\draw  (C3) node [right]{\scriptsize$A_3$};
\draw[thick,red] (B2)--(C3)--(B3)--(D2)--(C2)--(B1)--cycle;
\draw[thick,blue,densely dotted] (B2)--(B3)--(C2)--cycle;
\end{tikzpicture}
\caption{\label{fig:l=2}An example of $\chi(M_d(P)\setminus J)=2$ used in the proof of Proposition \ref{pro:every}. Here $P$ is the polygon with 6 (red) edges and $J$ is the set of 3 (blue) dotted  non-crossing diagonals. }
\end{figure}
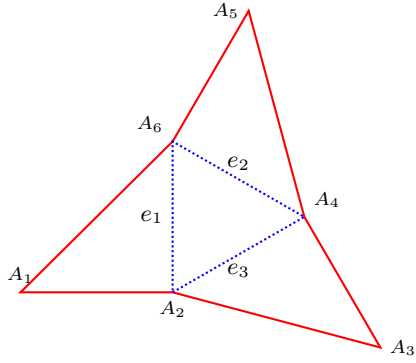

\begin{figure}[H]
\centering
\begin{tikzpicture}[scale=2]

\coordinate (B1) at (0,0);
\coordinate (C1) at (0.5,-0.866);
\coordinate (D1) at (1.366,-1.366);

\coordinate (B2) at (1,0);
\coordinate (C2) at (1,1);
\coordinate (D2) at (1+0.5,1.866);
\coordinate (e1) at (1,0.5); \draw  (e1) node [ left]{\small$e_1$};
\coordinate (e2) at (1/2+1/2+0.866/2,1/2+0.5/2); \draw  (e2) node [ above ]{\small$e_2$};
\coordinate (e3) at (1/2+1/2+0.866/2,0.5/2); \draw  (e3) node [ below ]{\small$e_3$};
\coordinate (B3) at (1+0.866,0.5);
\coordinate (C3) at (1+0.866+0.5,0.5-0.866);
\coordinate (D3) at (1+0.866+1.366,0.5-1.366);
\coordinate (e4) at (1/2+0.866/2+1/2+0.866/2+0.5/2,0.5/2+ 0.5/2-0.866/2);\draw  (e4) node [  left ]{\small$e_4$};
\coordinate (e5) at (1/2+0.866/2+1/2+1/2+0.866/2+0.5/2,0.5/2+ 0.5/2-0.866/2);\draw  (e5) node [  right ]{\small$e_5$};
\coordinate (e6) at (1/2+0.866/2+1/2+1/2+0.866/2,0.5/2+ 0.5/2);\draw  (e6) node [  above ]{\small$e_6$};
\coordinate (B4) at (1+0.866+1,0.5);
\coordinate (C4) at (1+0.866+1,0.5+1);
\coordinate (D4) at (1+0.866+1+0.5,0.5+1.866);

\coordinate (B5) at (1+0.866+1+0.866,0.5+0.5);
\coordinate (C5) at (1+0.866+1+0.866+0.5,0.5+0.5-0.866);
\coordinate (D5) at (1+0.866+1+0.866+1.366,0.5+0.5-1.366);

\coordinate (B6) at (1+0.866+1+0.866+1,0.5+0.5);
\coordinate (C6) at (1+0.866+1+0.866+1,0.5+0.5+1);
\coordinate (D6) at (1+0.866+1+0.866+1+0.5,0.5+0.5+1.866);

\coordinate (B7) at (1+0.866+1+0.866+1+0.866,0.5+0.5+0.5);
\coordinate (B8) at (1+0.866+1+0.866+1+0.866+1,0.5+0.5+0.5);
\coordinate (B9) at (1+0.866+1+0.866+1+0.866+1+0.866,0.5+0.5+0.5+0.5);

\coordinate (C) at (2*\pgfmathresult, 1);

\draw (B1)node [above]{\scriptsize $A_1$};
\draw(B2)node [below]{\scriptsize$A_2$} ;
\draw(B3)node [below right]{\scriptsize$A_7$} ;
\draw(B4)node [right]{\scriptsize$A_5$}  ;

\draw  (C2) node [ left]{\scriptsize$A_9$}; \draw (D2) node [ left]{\scriptsize$A_8$};
\draw  (C3) node [below left]{\scriptsize$A_3$} ; \draw(D3) node [right]{\scriptsize$A_4$};
\draw  (C4) node [ left]{\scriptsize$A_6$} ;

\draw[thick,red] (B2)--(C3)--(D3)--(B4)--(C4)--(B3)--(D2)--(C2)--(B1)--cycle;
\draw[thick,blue,densely dotted] (B2)--(C2)--(B3)--cycle;
\draw[thick,blue,densely dotted] (C3)--(B4)--(B3)--cycle;
\end{tikzpicture}
\caption{\label{fig:l=3}An example of $\chi(M_d(P)\setminus J)=3$ used in the proof of Proposition \ref{pro:every}. Here $P$ is the polygon with 9 (red) edges and $J$ is the set of 6 (blue) dotted  non-crossing diagonals. }
\end{figure}
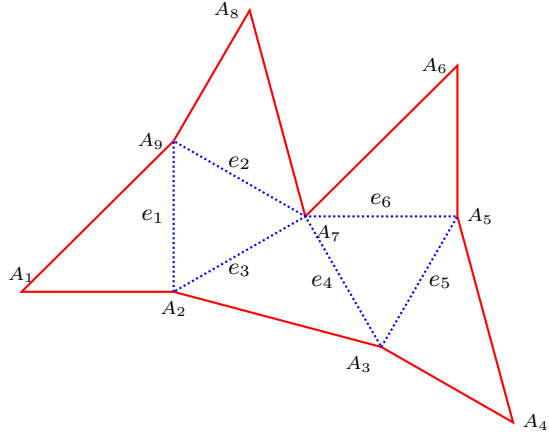

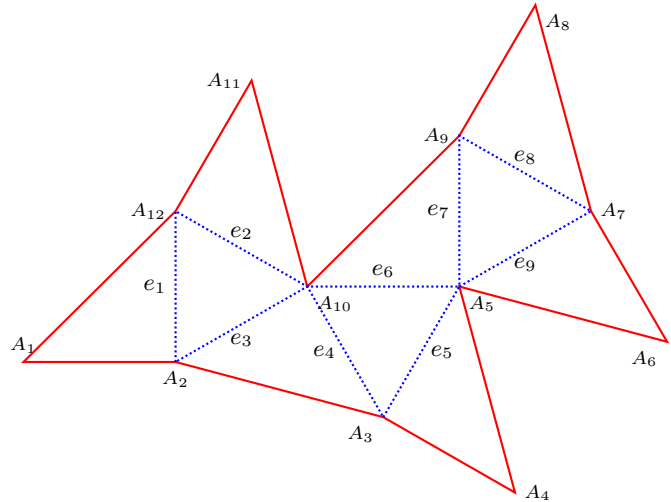
\begin{figure}[H]
\centering
\begin{tikzpicture}[scale=2]

\coordinate (B1) at (0,0);
\coordinate (C1) at (0.5,-0.866);
\coordinate (D1) at (1.366,-1.366);

\coordinate (B2) at (1,0);
\coordinate (C2) at (1,1);
\coordinate (D2) at (1+0.5,1.866);
\coordinate (e1) at (1,0.5); \draw  (e1) node [ left]{\small$e_1$};
\coordinate (e2) at (1/2+1/2+0.866/2,1/2+0.5/2); \draw  (e2) node [ above ]{\small$e_2$};
\coordinate (e3) at (1/2+1/2+0.866/2,0.5/2); \draw  (e3) node [ below ]{\small$e_3$};
\coordinate (B3) at (1+0.866,0.5);
\coordinate (C3) at (1+0.866+0.5,0.5-0.866);
\coordinate (D3) at (1+0.866+1.366,0.5-1.366);
\coordinate (e4) at (1/2+0.866/2+1/2+0.866/2+0.5/2,0.5/2+ 0.5/2-0.866/2);\draw  (e4) node [  left ]{\small$e_4$};
\coordinate (e5) at (1/2+0.866/2+1/2+1/2+0.866/2+0.5/2,0.5/2+ 0.5/2-0.866/2);\draw  (e5) node [  right ]{\small$e_5$};
\coordinate (e6) at (1/2+0.866/2+1/2+1/2+0.866/2,0.5/2+ 0.5/2);\draw  (e6) node [  above ]{\small$e_6$};
\coordinate (B4) at (1+0.866+1,0.5);
\coordinate (C4) at (1+0.866+1,0.5+1);
\coordinate (D4) at (1+0.866+1+0.5,0.5+1.866);
\coordinate (e7) at (1+1.866,0.5+0.5); \draw  (e7) node [ left]{\small$e_7$};
\coordinate (e8) at (1/2+1/2+0.866/2+1.866,1/2+0.5/2+0.5); \draw  (e8) node [ above ]{\small$e_8$};
\coordinate (e9) at (1/2+1/2+0.866/2+1.866,0.5/2+0.5); \draw  (e9) node [ below ]{\small$e_9$};
\coordinate (B5) at (1+0.866+1+0.866,0.5+0.5);
\coordinate (C5) at (1+0.866+1+0.866+0.5,0.5+0.5-0.866);
\coordinate (D5) at (1+0.866+1+0.866+1.366,0.5+0.5-1.366);

\coordinate (B6) at (1+0.866+1+0.866+1,0.5+0.5);
\coordinate (C6) at (1+0.866+1+0.866+1,0.5+0.5+1);
\coordinate (D6) at (1+0.866+1+0.866+1+0.5,0.5+0.5+1.866);

\coordinate (B7) at (1+0.866+1+0.866+1+0.866,0.5+0.5+0.5);
\coordinate (B8) at (1+0.866+1+0.866+1+0.866+1,0.5+0.5+0.5);
\coordinate (B9) at (1+0.866+1+0.866+1+0.866+1+0.866,0.5+0.5+0.5+0.5);

\coordinate (C) at (2*\pgfmathresult, 1);

\draw (B1)node [above]{\scriptsize $A_1$} ;
\draw (B2)node [below]{\scriptsize$A_2$};
\draw (B3)node [below right]{\scriptsize$A_{10}$};
\draw (B4)node [below right]{\scriptsize$A_5$};
\draw (B5)node [right]{\scriptsize$A_7$};
\draw  (C2) node [ left]{\scriptsize$A_{12}$};\draw(D2) node [ left]{\scriptsize$A_{11}$};
\draw  (C3) node [below left]{\scriptsize$A_3$};\draw (D3) node [right]{\scriptsize$A_4$};
\draw  (C4) node [ left]{\scriptsize$A_9$};\draw (D4) node [below right]{\scriptsize$A_8$};
\draw  (C5) node [below left]{\scriptsize$A_6$};

\draw[thick,red] (B2)--(C3)--(D3)--(B4)--(C5)--(B5)--(D4)--(C4)--(B3)--(D2)--(C2)--(B1)--cycle;
\draw[thick,blue,densely dotted] (B2)--(C2)--(B3)--cycle;
\draw[thick,blue,densely dotted] (C3)--(B4)--(B3)--cycle;
\draw[thick,blue,densely dotted] (B4)--(B5)--(C4)--cycle;
\end{tikzpicture}
\caption{\label{fig:l=4}An example of $\chi(M_d(P)\setminus J)=4$ used in the proof of Proposition \ref{pro:every}. Here $P$ is the polygon with 12 (red) edges and $J$ is the set of 9 (blue) dotted  non-crossing diagonals. }
\end{figure}

Construct $P$ with $|P|=3|l|$ and linearly order the vertices of $P$ in counter-clockwise direction, $A_1,A_2,\ldots,A_{3|l|}$ (see Figs.~\ref{fig:l=2},\ref{fig:l=3},\ref{fig:l=4},\ref{fig:l=5678} for $l=2,3,4,5,6,7,8$, respectively). We refer readers to Appendix for the detailed information of such polygons. Set $$X=\{\{3k+2,3|l|-3k,3|l|-3k-2\}:0\leq k\leq {|l|}/{2}-1,k\in\mathbb{Z}\}\cup\{\{3k+2,3k,3|l|-3k+1\}:1\leq k< |l|/2,k\in\mathbb{Z}\}$$
and $J=\{A_iA_j:\{i,j,k\}\in X\}$. Now we label the diagonals as
$$\begin{array}{llll}
e_{6k+1}=A_{3k+2}A_{3|l|-3k},& e_{6k+2}=A_{3|l|-3k-2}A_{3l-3k},& e_{6k+3}=A_{3k+2}A_{3|l|-3k-2},& \forall\ 0\leq k\leq {|l|}/{2}-1,k\in\mathbb{Z};\\
e_{6k-2}=A_{3k}A_{3|l|-3k+1},& e_{6k-1}=A_{3k+2}A_{3k},& e_{6k}=A_{3k+2}A_{3|l|-3k+1},& \forall\ 1\leq k< {|l|}/{2},k\in\mathbb{Z}.\end{array}$$
Then $J=\{e_1,\cdots,e_{3(|l|-1)}\}$, and we can check that for $I\subset J$, $I\in NC_c[J]$ if and only if $$I\cap\{e_k,e_{k+1}\}\neq \varnothing,\ \forall\ 1\leq k\leq 3|l|-4 \ \ \text{and}\ \  I\cap\{e_{3k-2},e_{3k}\}\neq \varnothing,\ \forall\ 1\leq k\leq |l|-1.$$ Set $J_k=\{e_{k+1},\cdots,e_{3(|l|-1)}\},\ 0\leq k<3(|l|-1),\ J_{3(|l|-1)}=\varnothing$.
Given $0< k\leq3(|l|-1)$, for any $I$ satisfying $J_{k-1}\subset I\subset J$, one can verify that
$$\begin{array}{cccccc}
I\setminus\{e_k\}\in NC_c[J]&\Leftrightarrow& e_{k-1}\in I\in NC_c[J]&\Leftrightarrow& J_{k-2}\subset I\in NC_c[J],& \text{ if } 3\nmid k;\\
I\setminus\{e_k\}\in NC_c[J]&\Leftrightarrow& \{e_{k-1},e_{k-2}\}\subset I\in NC_c[J]&\Leftrightarrow& J_{k-3}\subset I\in NC_c[J],& \text{ if } 3\mid k.
\end{array}$$

 Set $$a_k=\sum_{J_{k}\subset I\in NC_c[J]} (-1)^{\#I},\ 0\leq k\leq3(|l|-1).$$
 Then $a_0=1,\ a_1=0$ and for $ 2\leq k\leq3(|l|-1)$, we have
 \begin{align*}
 a_k&=a_{k-1}+\sum_{J_{k-1}\subset I,I\setminus\{e_k\}\in NC_c[J]} (-1)^{\#(I\setminus\{e_k\})}
 \\&=a_{k-1}+\left\{\begin{array}{ll}\sum\limits_{J_{k-2}\subset I\in NC_c[J]} (-1)^{\#(I\setminus\{e_k\})}, &3\nmid k,\\ \sum\limits_{J_{k-3}\subset I\in NC_c[J]} (-1)^{\#(I\setminus\{e_k\})}, &3\mid k,\end{array}\right.
 \\&=\left\{\begin{array}{ll}a_{k-1}-a_{k-2}, &3\nmid k,\\ a_{k-1}-a_{k-3}, &3\mid k.\end{array}\right.
 \end{align*}
 Using this formula by induction we have $$a_{3k}=(-1)^{k}(k+1),\ \forall\ 0\leq k\leq |l|;\ \ a_{3k+1}=(-1)^{k}k,\ a_{3k+2}=(-1)^{k+1},\ \forall\ 0\leq k< |l|.$$ Therefore \begin{align*}
\chi(M_d(P)\setminus J)=(-1)^{|P|+1}\sum_{I\in NC_c[J]} (-1)^{\#I}
=(-1)^{3|l|+1}a_{3(|l|-1)}=(-1)^{3|l|+1}(-1)^{|l|-1}
|l|=|l|.
\end{align*}
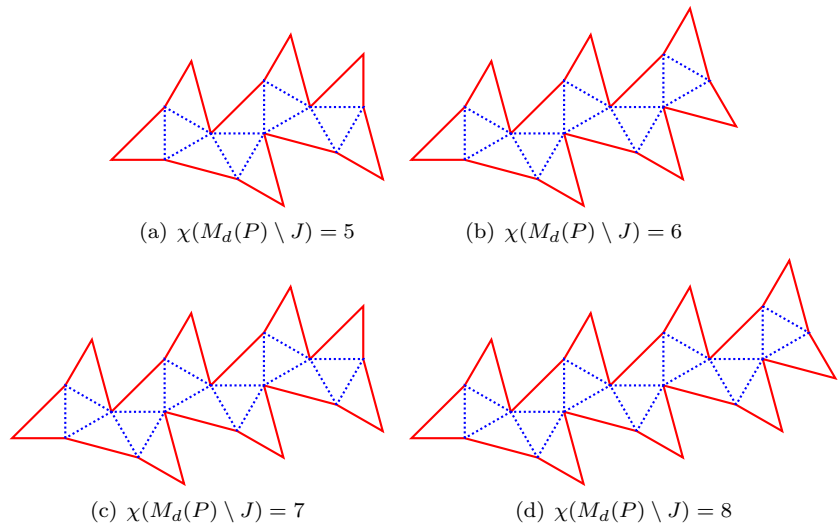
\begin{figure}[H]
\centering
\subfigure[$\chi(M_d(P)\setminus J)=5$]{
\begin{tikzpicture}[scale=0.7]
\coordinate (B1) at (0,0);
\coordinate (C1) at (0.5,-0.866);
\coordinate (D1) at (1.366,-1.366);

\coordinate (B2) at (1,0);
\coordinate (C2) at (1,1);
\coordinate (D2) at (1+0.5,1.866);

\coordinate (B3) at (1+0.866,0.5);
\coordinate (C3) at (1+0.866+0.5,0.5-0.866);
\coordinate (D3) at (1+0.866+1.366,0.5-1.366);

\coordinate (B4) at (1+0.866+1,0.5);
\coordinate (C4) at (1+0.866+1,0.5+1);
\coordinate (D4) at (1+0.866+1+0.5,0.5+1.866);

\coordinate (B5) at (1+0.866+1+0.866,0.5+0.5);
\coordinate (C5) at (1+0.866+1+0.866+0.5,0.5+0.5-0.866);
\coordinate (D5) at (1+0.866+1+0.866+1.366,0.5+0.5-1.366);

\coordinate (B6) at (1+0.866+1+0.866+1,0.5+0.5);
\coordinate (C6) at (1+0.866+1+0.866+1,0.5+0.5+1);
\coordinate (D6) at (1+0.866+1+0.866+1+0.5,0.5+0.5+1.866);

\coordinate (B7) at (1+0.866+1+0.866+1+0.866,0.5+0.5+0.5);
\coordinate (B8) at (1+0.866+1+0.866+1+0.866+1,0.5+0.5+0.5);
\coordinate (B9) at (1+0.866+1+0.866+1+0.866+1+0.866,0.5+0.5+0.5+0.5);


\draw[thick,red] (B2)--(C3)--(D3)--(B4)--(C5) -- (D5)--(B6)--(C6)--(B5)--(D4)--(C4)--(B3)--(D2)--(C2)--(B1)--cycle;
\draw[densely dotted,thick,blue] (B2)--(C2)--(B3)--cycle;
\draw[densely dotted,thick,blue] (C3)--(B4)--(B3)--cycle;
\draw[densely dotted,thick,blue] (B4)--(B5)--(C4)--cycle;
\draw[densely dotted,thick,blue] (C5)--(B5)--(B6)--cycle;
\end{tikzpicture}
}
\subfigure[$\chi(M_d(P)\setminus J)=6$]{
\begin{tikzpicture}[scale=0.7]
\coordinate (B1) at (0,0);
\coordinate (C1) at (0.5,-0.866);
\coordinate (D1) at (1.366,-1.366);

\coordinate (B2) at (1,0);
\coordinate (C2) at (1,1);
\coordinate (D2) at (1+0.5,1.866);

\coordinate (B3) at (1+0.866,0.5);
\coordinate (C3) at (1+0.866+0.5,0.5-0.866);
\coordinate (D3) at (1+0.866+1.366,0.5-1.366);

\coordinate (B4) at (1+0.866+1,0.5);
\coordinate (C4) at (1+0.866+1,0.5+1);
\coordinate (D4) at (1+0.866+1+0.5,0.5+1.866);

\coordinate (B5) at (1+0.866+1+0.866,0.5+0.5);
\coordinate (C5) at (1+0.866+1+0.866+0.5,0.5+0.5-0.866);
\coordinate (D5) at (1+0.866+1+0.866+1.366,0.5+0.5-1.366);

\coordinate (B6) at (1+0.866+1+0.866+1,0.5+0.5);
\coordinate (C6) at (1+0.866+1+0.866+1,0.5+0.5+1);
\coordinate (D6) at (1+0.866+1+0.866+1+0.5,0.5+0.5+1.866);

\coordinate (B7) at (1+0.866+1+0.866+1+0.866,0.5+0.5+0.5);
\coordinate (C7) at (1+0.866+1+0.866+1+0.866+0.5,0.5+0.5+0.5-0.866);
\coordinate (D7) at (1+0.866+1+0.866+1+0.866+1.366,0.5+0.5+0.5-1.366);

\coordinate (B8) at (1+0.866+1+0.866+1+0.866+1,0.5+0.5+0.5);
\coordinate (C8) at (1+0.866+1+0.866+1+0.866+1,0.5+0.5+0.5+1);
\coordinate (D8) at (1+0.866+1+0.866+1+0.866+1+0.5,0.5+0.5+0.5+1.866);

\coordinate (B9) at (1+0.866+1+0.866+1+0.866+1+0.866,0.5+0.5+0.5+0.5);


\draw[thick,red] (B2)--(C3)--(D3)--(B4)--(C5) -- (D5)--(B6)--(C7)--(B7)--(D6)--(C6)--(B5)--(D4)--(C4)--(B3)--(D2)--(C2)--(B1)--cycle;
\draw[densely dotted,thick,blue] (B2)--(C2)--(B3)--cycle;
\draw[densely dotted,thick,blue] (C3)--(B4)--(B3)--cycle;
\draw[densely dotted,thick,blue] (B4)--(B5)--(C4)--cycle;
\draw[densely dotted,thick,blue] (C5)--(B5)--(B6)--cycle;
\draw[densely dotted,thick,blue] (C6)--(B7)--(B6)--cycle;
\end{tikzpicture}
}

\subfigure[$\chi(M_d(P)\setminus J)=7$]{
\begin{tikzpicture}[scale=0.7]
\coordinate (B1) at (0,0);
\coordinate (C1) at (0.5,-0.866);
\coordinate (D1) at (1.366,-1.366);

\coordinate (B2) at (1,0);
\coordinate (C2) at (1,1);
\coordinate (D2) at (1+0.5,1.866);

\coordinate (B3) at (1+0.866,0.5);
\coordinate (C3) at (1+0.866+0.5,0.5-0.866);
\coordinate (D3) at (1+0.866+1.366,0.5-1.366);

\coordinate (B4) at (1+0.866+1,0.5);
\coordinate (C4) at (1+0.866+1,0.5+1);
\coordinate (D4) at (1+0.866+1+0.5,0.5+1.866);

\coordinate (B5) at (1+0.866+1+0.866,0.5+0.5);
\coordinate (C5) at (1+0.866+1+0.866+0.5,0.5+0.5-0.866);
\coordinate (D5) at (1+0.866+1+0.866+1.366,0.5+0.5-1.366);

\coordinate (B6) at (1+0.866+1+0.866+1,0.5+0.5);
\coordinate (C6) at (1+0.866+1+0.866+1,0.5+0.5+1);
\coordinate (D6) at (1+0.866+1+0.866+1+0.5,0.5+0.5+1.866);

\coordinate (B7) at (1+0.866+1+0.866+1+0.866,0.5+0.5+0.5);
\coordinate (C7) at (1+0.866+1+0.866+1+0.866+0.5,0.5+0.5+0.5-0.866);
\coordinate (D7) at (1+0.866+1+0.866+1+0.866+1.366,0.5+0.5+0.5-1.366);

\coordinate (B8) at (1+0.866+1+0.866+1+0.866+1,0.5+0.5+0.5);
\coordinate (C8) at (1+0.866+1+0.866+1+0.866+1,0.5+0.5+0.5+1);
\coordinate (D8) at (1+0.866+1+0.866+1+0.866+1+0.5,0.5+0.5+0.5+1.866);

\coordinate (B9) at (1+0.866+1+0.866+1+0.866+1+0.866,0.5+0.5+0.5+0.5);


\draw[thick,red] (B2)--(C3)--(D3)--(B4)--(C5) -- (D5)--(B6)--(C7)--(D7)--(B8)--(C8)--(B7)--(D6)--(C6)--(B5)--(D4)--(C4)--(B3)--(D2)--(C2)--(B1)--cycle;
\draw[densely dotted,thick,blue] (B2)--(C2)--(B3)--cycle;
\draw[densely dotted,thick,blue] (C3)--(B4)--(B3)--cycle;
\draw[densely dotted,thick,blue] (B4)--(B5)--(C4)--cycle;
\draw[densely dotted,thick,blue] (C5)--(B5)--(B6)--cycle;
\draw[densely dotted,thick,blue] (C6)--(B7)--(B6)--cycle;
\draw[densely dotted,thick,blue] (C7)--(B7)--(B8)--cycle;
\end{tikzpicture}
}
\subfigure[$\chi(M_d(P)\setminus J)=8$]{
\begin{tikzpicture}[scale=0.7]
\coordinate (B1) at (0,0);
\coordinate (C1) at (0.5,-0.866);
\coordinate (D1) at (1.366,-1.366);

\coordinate (B2) at (1,0);
\coordinate (C2) at (1,1);
\coordinate (D2) at (1+0.5,1.866);

\coordinate (B3) at (1+0.866,0.5);
\coordinate (C3) at (1+0.866+0.5,0.5-0.866);
\coordinate (D3) at (1+0.866+1.366,0.5-1.366);

\coordinate (B4) at (1+0.866+1,0.5);
\coordinate (C4) at (1+0.866+1,0.5+1);
\coordinate (D4) at (1+0.866+1+0.5,0.5+1.866);

\coordinate (B5) at (1+0.866+1+0.866,0.5+0.5);
\coordinate (C5) at (1+0.866+1+0.866+0.5,0.5+0.5-0.866);
\coordinate (D5) at (1+0.866+1+0.866+1.366,0.5+0.5-1.366);

\coordinate (B6) at (1+0.866+1+0.866+1,0.5+0.5);
\coordinate (C6) at (1+0.866+1+0.866+1,0.5+0.5+1);
\coordinate (D6) at (1+0.866+1+0.866+1+0.5,0.5+0.5+1.866);

\coordinate (B7) at (1+0.866+1+0.866+1+0.866,0.5+0.5+0.5);
\coordinate (C7) at (1+0.866+1+0.866+1+0.866+0.5,0.5+0.5+0.5-0.866);
\coordinate (D7) at (1+0.866+1+0.866+1+0.866+1.366,0.5+0.5+0.5-1.366);

\coordinate (B8) at (1+0.866+1+0.866+1+0.866+1,0.5+0.5+0.5);
\coordinate (C8) at (1+0.866+1+0.866+1+0.866+1,0.5+0.5+0.5+1);
\coordinate (D8) at (1+0.866+1+0.866+1+0.866+1+0.5,0.5+0.5+0.5+1.866);

\coordinate (B9) at (1+0.866+1+0.866+1+0.866+1+0.866,0.5+0.5+0.5+0.5);
\coordinate (C9) at (1+0.866+1+0.866+1+0.866+1+0.866+0.5,0.5+0.5+0.5+0.5-0.866);


\draw[thick,red] (B2)--(C3)--(D3)--(B4)--(C5) -- (D5)--(B6)--(C7)--(D7)--(B8)--(C9)--(B9)--(D8)--(C8)--(B7)--(D6)--(C6)--(B5)--(D4)--(C4)--(B3)--(D2)--(C2)--(B1)--cycle;
\draw[densely dotted,thick,blue] (B2)--(C2)--(B3)--cycle;
\draw[densely dotted,thick,blue] (C3)--(B4)--(B3)--cycle;
\draw[densely dotted,thick,blue] (B4)--(B5)--(C4)--cycle;
\draw[densely dotted,thick,blue] (C5)--(B5)--(B6)--cycle;
\draw[densely dotted,thick,blue] (C6)--(B7)--(B6)--cycle;
\draw[densely dotted,thick,blue] (C7)--(B7)--(B8)--cycle;
\draw[densely dotted,thick,blue] (C8)--(B9)--(B8)--cycle;
\end{tikzpicture}
}
\caption{\label{fig:l=5678}Examples of $\chi(M_d(P)\setminus J)\in\{5,6,7,8\}$ used in the proof of Proposition \ref{pro:every}. Here $P$ is the polygon with red edges and $J$ is the set of corresponding blue non-crossing diagonals. }
\end{figure}

For the case of $l<-1$, we consider the polygon $P'$ (see Fig.~\ref{fig:l=-2345678}) with linearly ordered vertices in counter-clockwise direction,  $A_1',A_2,\ldots,A_{3|l|}, A_{3|l|+1}'$, such that $A_1'$ is in the segment $A_1A_2,$ $A_{3|l|+1}'$ is in the segment $A_1A_{3|l|}$, where$A_1,A_2,\ldots,A_{3|l|}$ are the vertices of $P$ above. Then $J\in NC[M_d(P')]$ and
$$\chi(M_d(P')\setminus J)=(-1)^{|P'|+1}\sum_{I\in NC_c[J]} (-1)^{\#I}
=(-1)^{3|l|+1+1}a_{3(|l|-1)}=(-1)^{3|l|}(-1)^{|l|-1}|l|=-|l|=l.$$
This completes the proof of Proposition \ref{pro:every}.

\begin{figure}[H]
\centering
\subfigure[$l=-2$]{
\begin{tikzpicture}[scale=0.7]
\coordinate (P) at (0.5,0);
\coordinate (Q) at (0.5,0.5);

\coordinate (C1) at (0.5,-0.866);
\coordinate (D1) at (1.366,-1.366);

\coordinate (B2) at (1,0);
\coordinate (C2) at (1,1);
\coordinate (D2) at (1+0.5,1.866);

\coordinate (B3) at (1+0.866,0.5);
\coordinate (C3) at (1+0.866+0.5,0.5-0.866);
\coordinate (D3) at (1+0.866+1.366,0.5-1.366);

\coordinate (B4) at (1+0.866+1,0.5);
\coordinate (C4) at (1+0.866+1,0.5+1);
\coordinate (D4) at (1+0.866+1+0.5,0.5+1.866);

\coordinate (B5) at (1+0.866+1+0.866,0.5+0.5);
\coordinate (C5) at (1+0.866+1+0.866+0.5,0.5+0.5-0.866);
\coordinate (D5) at (1+0.866+1+0.866+1.366,0.5+0.5-1.366);

\coordinate (B6) at (1+0.866+1+0.866+1,0.5+0.5);
\coordinate (C6) at (1+0.866+1+0.866+1,0.5+0.5+1);
\coordinate (D6) at (1+0.866+1+0.866+1+0.5,0.5+0.5+1.866);

\coordinate (B7) at (1+0.866+1+0.866+1+0.866,0.5+0.5+0.5);
\coordinate (B8) at (1+0.866+1+0.866+1+0.866+1,0.5+0.5+0.5);
\coordinate (B9) at (1+0.866+1+0.866+1+0.866+1+0.866,0.5+0.5+0.5+0.5);

\draw[thick,red] (B2)--(C3)--(B3)--(D2)--(C2)--(Q)--(P)--cycle;
\draw[densely dotted,thick,blue] (B2)--(C2)--(B3)--cycle;
\end{tikzpicture}
}
\subfigure[$l=-3$]{
\begin{tikzpicture}[scale=0.7]
\coordinate (P) at (0.5,0);
\coordinate (Q) at (0.5,0.5);

\coordinate (C1) at (0.5,-0.866);
\coordinate (D1) at (1.366,-1.366);

\coordinate (B2) at (1,0);
\coordinate (C2) at (1,1);
\coordinate (D2) at (1+0.5,1.866);

\coordinate (B3) at (1+0.866,0.5);
\coordinate (C3) at (1+0.866+0.5,0.5-0.866);
\coordinate (D3) at (1+0.866+1.366,0.5-1.366);

\coordinate (B4) at (1+0.866+1,0.5);
\coordinate (C4) at (1+0.866+1,0.5+1);
\coordinate (D4) at (1+0.866+1+0.5,0.5+1.866);

\coordinate (B5) at (1+0.866+1+0.866,0.5+0.5);
\coordinate (C5) at (1+0.866+1+0.866+0.5,0.5+0.5-0.866);
\coordinate (D5) at (1+0.866+1+0.866+1.366,0.5+0.5-1.366);

\coordinate (B6) at (1+0.866+1+0.866+1,0.5+0.5);
\coordinate (C6) at (1+0.866+1+0.866+1,0.5+0.5+1);
\coordinate (D6) at (1+0.866+1+0.866+1+0.5,0.5+0.5+1.866);

\coordinate (B7) at (1+0.866+1+0.866+1+0.866,0.5+0.5+0.5);
\coordinate (B8) at (1+0.866+1+0.866+1+0.866+1,0.5+0.5+0.5);
\coordinate (B9) at (1+0.866+1+0.866+1+0.866+1+0.866,0.5+0.5+0.5+0.5);

\draw[thick,red] (B2)--(C3)--(D3)--(B4)--(C4)--(B3)--(D2)--(C2)--(Q)--(P)--cycle;
\draw[densely dotted,thick,blue] (B2)--(C2)--(B3)--cycle;
\draw[densely dotted,thick,blue] (C3)--(B4)--(B3)--cycle;
\end{tikzpicture}
}
\subfigure[$l=-4$]{
\begin{tikzpicture}[scale=0.7]
\coordinate (P) at (0.5,0);
\coordinate (Q) at (0.5,0.5);

\coordinate (C1) at (0.5,-0.866);
\coordinate (D1) at (1.366,-1.366);

\coordinate (B2) at (1,0);
\coordinate (C2) at (1,1);
\coordinate (D2) at (1+0.5,1.866);

\coordinate (B3) at (1+0.866,0.5);
\coordinate (C3) at (1+0.866+0.5,0.5-0.866);
\coordinate (D3) at (1+0.866+1.366,0.5-1.366);

\coordinate (B4) at (1+0.866+1,0.5);
\coordinate (C4) at (1+0.866+1,0.5+1);
\coordinate (D4) at (1+0.866+1+0.5,0.5+1.866);

\coordinate (B5) at (1+0.866+1+0.866,0.5+0.5);
\coordinate (C5) at (1+0.866+1+0.866+0.5,0.5+0.5-0.866);
\coordinate (D5) at (1+0.866+1+0.866+1.366,0.5+0.5-1.366);

\coordinate (B6) at (1+0.866+1+0.866+1,0.5+0.5);
\coordinate (C6) at (1+0.866+1+0.866+1,0.5+0.5+1);
\coordinate (D6) at (1+0.866+1+0.866+1+0.5,0.5+0.5+1.866);

\coordinate (B7) at (1+0.866+1+0.866+1+0.866,0.5+0.5+0.5);
\coordinate (B8) at (1+0.866+1+0.866+1+0.866+1,0.5+0.5+0.5);
\coordinate (B9) at (1+0.866+1+0.866+1+0.866+1+0.866,0.5+0.5+0.5+0.5);

\draw[thick,red] (B2)--(C3)--(D3)--(B4)--(C5) -- (B5)--(D4)--(C4)--(B3)--(D2)--(C2)--(Q)--(P)--cycle;
\draw[densely dotted,thick,blue] (B2)--(C2)--(B3)--cycle;
\draw[densely dotted,thick,blue] (C3)--(B4)--(B3)--cycle;
\draw[densely dotted,thick,blue] (B4)--(B5)--(C4)--cycle;
\end{tikzpicture}
}
\subfigure[$l=-5$]{
\begin{tikzpicture}[scale=0.7]
\coordinate (P) at (0.5,0);
\coordinate (Q) at (0.5,0.5);

\coordinate (C1) at (0.5,-0.866);
\coordinate (D1) at (1.366,-1.366);

\coordinate (B2) at (1,0);
\coordinate (C2) at (1,1);
\coordinate (D2) at (1+0.5,1.866);

\coordinate (B3) at (1+0.866,0.5);
\coordinate (C3) at (1+0.866+0.5,0.5-0.866);
\coordinate (D3) at (1+0.866+1.366,0.5-1.366);

\coordinate (B4) at (1+0.866+1,0.5);
\coordinate (C4) at (1+0.866+1,0.5+1);
\coordinate (D4) at (1+0.866+1+0.5,0.5+1.866);

\coordinate (B5) at (1+0.866+1+0.866,0.5+0.5);
\coordinate (C5) at (1+0.866+1+0.866+0.5,0.5+0.5-0.866);
\coordinate (D5) at (1+0.866+1+0.866+1.366,0.5+0.5-1.366);

\coordinate (B6) at (1+0.866+1+0.866+1,0.5+0.5);
\coordinate (C6) at (1+0.866+1+0.866+1,0.5+0.5+1);
\coordinate (D6) at (1+0.866+1+0.866+1+0.5,0.5+0.5+1.866);

\coordinate (B7) at (1+0.866+1+0.866+1+0.866,0.5+0.5+0.5);
\coordinate (B8) at (1+0.866+1+0.866+1+0.866+1,0.5+0.5+0.5);
\coordinate (B9) at (1+0.866+1+0.866+1+0.866+1+0.866,0.5+0.5+0.5+0.5);

\draw[thick,red] (B2)--(C3)--(D3)--(B4)--(C5) -- (D5)--(B6)--(C6)--(B5)--(D4)--(C4)--(B3)--(D2)--(C2)--(Q)--(P)--cycle;
\draw[densely dotted,thick,blue] (B2)--(C2)--(B3)--cycle;
\draw[densely dotted,thick,blue] (C3)--(B4)--(B3)--cycle;
\draw[densely dotted,thick,blue] (B4)--(B5)--(C4)--cycle;
\draw[densely dotted,thick,blue] (C5)--(B5)--(B6)--cycle;
\end{tikzpicture}
}
\subfigure[$l=-6$]{
\begin{tikzpicture}[scale=0.7]
\coordinate (P) at (0.5,0);
\coordinate (Q) at (0.5,0.5);
\coordinate (C1) at (0.5,-0.866);
\coordinate (D1) at (1.366,-1.366);

\coordinate (B2) at (1,0);
\coordinate (C2) at (1,1);
\coordinate (D2) at (1+0.5,1.866);

\coordinate (B3) at (1+0.866,0.5);
\coordinate (C3) at (1+0.866+0.5,0.5-0.866);
\coordinate (D3) at (1+0.866+1.366,0.5-1.366);

\coordinate (B4) at (1+0.866+1,0.5);
\coordinate (C4) at (1+0.866+1,0.5+1);
\coordinate (D4) at (1+0.866+1+0.5,0.5+1.866);

\coordinate (B5) at (1+0.866+1+0.866,0.5+0.5);
\coordinate (C5) at (1+0.866+1+0.866+0.5,0.5+0.5-0.866);
\coordinate (D5) at (1+0.866+1+0.866+1.366,0.5+0.5-1.366);

\coordinate (B6) at (1+0.866+1+0.866+1,0.5+0.5);
\coordinate (C6) at (1+0.866+1+0.866+1,0.5+0.5+1);
\coordinate (D6) at (1+0.866+1+0.866+1+0.5,0.5+0.5+1.866);

\coordinate (B7) at (1+0.866+1+0.866+1+0.866,0.5+0.5+0.5);
\coordinate (C7) at (1+0.866+1+0.866+1+0.866+0.5,0.5+0.5+0.5-0.866);
\coordinate (D7) at (1+0.866+1+0.866+1+0.866+1.366,0.5+0.5+0.5-1.366);

\coordinate (B8) at (1+0.866+1+0.866+1+0.866+1,0.5+0.5+0.5);
\coordinate (C8) at (1+0.866+1+0.866+1+0.866+1,0.5+0.5+0.5+1);
\coordinate (D8) at (1+0.866+1+0.866+1+0.866+1+0.5,0.5+0.5+0.5+1.866);

\coordinate (B9) at (1+0.866+1+0.866+1+0.866+1+0.866,0.5+0.5+0.5+0.5);

\draw[thick,red] (B2)--(C3)--(D3)--(B4)--(C5) -- (D5)--(B6)--(C7)--(B7)--(D6)--(C6)--(B5)--(D4)--(C4)--(B3)--(D2)--(C2)--(Q)--(P)--cycle;
\draw[densely dotted,thick,blue] (B2)--(C2)--(B3)--cycle;
\draw[densely dotted,thick,blue] (C3)--(B4)--(B3)--cycle;
\draw[densely dotted,thick,blue] (B4)--(B5)--(C4)--cycle;
\draw[densely dotted,thick,blue] (C5)--(B5)--(B6)--cycle;
\draw[densely dotted,thick,blue] (C6)--(B7)--(B6)--cycle;
\end{tikzpicture}
}
\subfigure[$l=-7$]{
\begin{tikzpicture}[scale=0.7]
\coordinate (P) at (0.5,0);
\coordinate (Q) at (0.5,0.5);
\coordinate (C1) at (0.5,-0.866);
\coordinate (D1) at (1.366,-1.366);

\coordinate (B2) at (1,0);
\coordinate (C2) at (1,1);
\coordinate (D2) at (1+0.5,1.866);

\coordinate (B3) at (1+0.866,0.5);
\coordinate (C3) at (1+0.866+0.5,0.5-0.866);
\coordinate (D3) at (1+0.866+1.366,0.5-1.366);

\coordinate (B4) at (1+0.866+1,0.5);
\coordinate (C4) at (1+0.866+1,0.5+1);
\coordinate (D4) at (1+0.866+1+0.5,0.5+1.866);

\coordinate (B5) at (1+0.866+1+0.866,0.5+0.5);
\coordinate (C5) at (1+0.866+1+0.866+0.5,0.5+0.5-0.866);
\coordinate (D5) at (1+0.866+1+0.866+1.366,0.5+0.5-1.366);

\coordinate (B6) at (1+0.866+1+0.866+1,0.5+0.5);
\coordinate (C6) at (1+0.866+1+0.866+1,0.5+0.5+1);
\coordinate (D6) at (1+0.866+1+0.866+1+0.5,0.5+0.5+1.866);

\coordinate (B7) at (1+0.866+1+0.866+1+0.866,0.5+0.5+0.5);
\coordinate (C7) at (1+0.866+1+0.866+1+0.866+0.5,0.5+0.5+0.5-0.866);
\coordinate (D7) at (1+0.866+1+0.866+1+0.866+1.366,0.5+0.5+0.5-1.366);

\coordinate (B8) at (1+0.866+1+0.866+1+0.866+1,0.5+0.5+0.5);
\coordinate (C8) at (1+0.866+1+0.866+1+0.866+1,0.5+0.5+0.5+1);
\coordinate (D8) at (1+0.866+1+0.866+1+0.866+1+0.5,0.5+0.5+0.5+1.866);

\coordinate (B9) at (1+0.866+1+0.866+1+0.866+1+0.866,0.5+0.5+0.5+0.5);

\draw[thick,red] (B2)--(C3)--(D3)--(B4)--(C5) -- (D5)--(B6)--(C7)--(D7)--(B8)--(C8)--(B7)--(D6)--(C6)--(B5)--(D4)--(C4)--(B3)--(D2)--(C2)--(Q)--(P)--cycle;
\draw[densely dotted,thick,blue] (B2)--(C2)--(B3)--cycle;
\draw[densely dotted,thick,blue] (C3)--(B4)--(B3)--cycle;
\draw[densely dotted,thick,blue] (B4)--(B5)--(C4)--cycle;
\draw[densely dotted,thick,blue] (C5)--(B5)--(B6)--cycle;
\draw[densely dotted,thick,blue] (C6)--(B7)--(B6)--cycle;
\draw[densely dotted,thick,blue] (C7)--(B7)--(B8)--cycle;
\end{tikzpicture}
}
\subfigure[$l=-8$]{
\begin{tikzpicture}[scale=0.7]
\coordinate (P) at (0.5,0);
\coordinate (Q) at (0.5,0.5);
\coordinate (C1) at (0.5,-0.866);
\coordinate (D1) at (1.366,-1.366);

\coordinate (B2) at (1,0);
\coordinate (C2) at (1,1);
\coordinate (D2) at (1+0.5,1.866);

\coordinate (B3) at (1+0.866,0.5);
\coordinate (C3) at (1+0.866+0.5,0.5-0.866);
\coordinate (D3) at (1+0.866+1.366,0.5-1.366);

\coordinate (B4) at (1+0.866+1,0.5);
\coordinate (C4) at (1+0.866+1,0.5+1);
\coordinate (D4) at (1+0.866+1+0.5,0.5+1.866);

\coordinate (B5) at (1+0.866+1+0.866,0.5+0.5);
\coordinate (C5) at (1+0.866+1+0.866+0.5,0.5+0.5-0.866);
\coordinate (D5) at (1+0.866+1+0.866+1.366,0.5+0.5-1.366);

\coordinate (B6) at (1+0.866+1+0.866+1,0.5+0.5);
\coordinate (C6) at (1+0.866+1+0.866+1,0.5+0.5+1);
\coordinate (D6) at (1+0.866+1+0.866+1+0.5,0.5+0.5+1.866);

\coordinate (B7) at (1+0.866+1+0.866+1+0.866,0.5+0.5+0.5);
\coordinate (C7) at (1+0.866+1+0.866+1+0.866+0.5,0.5+0.5+0.5-0.866);
\coordinate (D7) at (1+0.866+1+0.866+1+0.866+1.366,0.5+0.5+0.5-1.366);

\coordinate (B8) at (1+0.866+1+0.866+1+0.866+1,0.5+0.5+0.5);
\coordinate (C8) at (1+0.866+1+0.866+1+0.866+1,0.5+0.5+0.5+1);
\coordinate (D8) at (1+0.866+1+0.866+1+0.866+1+0.5,0.5+0.5+0.5+1.866);

\coordinate (B9) at (1+0.866+1+0.866+1+0.866+1+0.866,0.5+0.5+0.5+0.5);
\coordinate (C9) at (1+0.866+1+0.866+1+0.866+1+0.866+0.5,0.5+0.5+0.5+0.5-0.866);

\draw[thick,red] (B2)--(C3)--(D3)--(B4)--(C5) -- (D5)--(B6)--(C7)--(D7)--(B8)--(C9)--(B9)--(D8)--(C8)--(B7)--(D6)--(C6)--(B5)--(D4)--(C4)--(B3)--(D2)--(C2)--
(Q)--(P)--cycle;
\draw[densely dotted,thick,blue] (B2)--(C2)--(B3)--cycle;
\draw[densely dotted,thick,blue] (C3)--(B4)--(B3)--cycle;
\draw[densely dotted,thick,blue] (B4)--(B5)--(C4)--cycle;
\draw[densely dotted,thick,blue] (C5)--(B5)--(B6)--cycle;
\draw[densely dotted,thick,blue] (C6)--(B7)--(B6)--cycle;
\draw[densely dotted,thick,blue] (C7)--(B7)--(B8)--cycle;
\draw[densely dotted,thick,blue] (C8)--(B9)--(B8)--cycle;
\end{tikzpicture}
}
\caption{\label{fig:l=-2345678}Examples of $\chi(M_d(P)\setminus J)=l\in\{-2,-3,\ldots,-8\}$ used in the proof of Proposition \ref{pro:every}. Here $P$ is the polygon with red edges and $J$ is the set of corresponding blue non-crossing diagonals. }
\end{figure}

\section{Proof of Theorem \ref{th:main1}}
\label{sec:thm4}

Let $J=\{A_iA_j:j\ne i-1,i,i+1\}$, and without loss of generality we let $i=1$ for simplicity.

(A) We suppose $\chi(M_d\setminus J)\ne 0$.

\textbf{Claim 1}\; $J\subset M_d$, i.e., every chord $A_1A_j$ is a diagonal of $P$, where $j\ne 1,2,n$.

Since $\chi(M_d\setminus J)\ne 0$, Proposition \ref{pro:1} implies that $J\cap M_d$ divides $P$ into convex polygons, with a common vertex $A_1$. Suppose that there exists $j_1\ne 1$ such that $A_1A_{j_1}\not\in M_d(P)$. Then  $A_1A_{j_1}$ is not an edge of these convex sub-polygons. Note that $A_{j_1}$ must be a vertex of a convex sub-polygon. So, $A_1A_{j_1}$ is a diagonal of such convex sub-polygon and thus $A_1A_{j_1}$ is a diagonal of $P$, which is a contradiction.

By Claim 1, we have $\angle A_3A_2A_1<\pi$ and $\angle A_1A_nA_{n-1}<\pi$ (see Fig.~\ref{fig:12n}).
\begin{figure}[H]\centering
\begin{tikzpicture}
\draw (20:3)--(0:3);
\draw (20:3)--(60:3);
\draw (60:3)--(90:2);
\draw (90:2)--(120:3);
\draw (120:3)--(150:3);
\draw (190:2)--(150:3);
\draw (190:2)--(210:3);
\draw (0:0) to (0:3);
\draw[dashed] (0:0) to (20:3);
\draw[dashed] (0:0) to (60:3);
\draw[red,dashed,very thick] (0:0) to (90:2);
\draw[dashed] (0:0) to (120:3);
\draw[dashed] (0:0) to (150:3);
\draw[red,dashed,very thick] (0:0) to (190:2);
\draw (0:0) to (210:3);
\node (A1) at (300:0.3) {$A_1$};
\node (A2) at (0:3.2) {$A_2$};
\node (An) at (210:3.2) {$A_n$};
\node (A3) at (20:3.2) {$A_3$};
\node (An-1) at (190:2.5) {$A_{n-1}$};
\end{tikzpicture}
\caption{\label{fig:12n}Illustration for the proof of Theorem \ref{th:main1} (A). In this polygon, all the dashed lines are collected in $J$, and all the (red) thick dashed lines are collected in $J_c$.}
\end{figure}
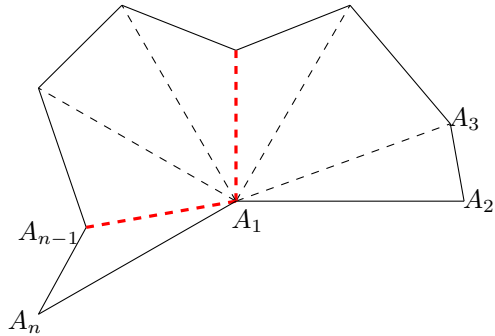

Let
$$J_c=\begin{cases}
\{A_1A_j:\angle A_{j-1}A_jA_{j+1}>\pi,j\ne 1,2,n\},& \exists j\ne 1,2,n\text{ s.t. }\angle A_{j-1}A_jA_{j+1}>\pi,\\
\varnothing,&\text{ otherwise}.
\end{cases}
$$

It is easy to see that $J_c\subset J$ (see Fig.~\ref{fig:12n}).

\textbf{Claim 2}\; If $I\subset J$ divides $P$ into convex polygons, then $I\supset J_c$.

Suppose the contrary, that $I\not\supset J_c$, i.e., there exists $A_1A_j\in J_c\setminus I$. Then there is a sub-polygon containing the vertices $A_{j-1},A_j,A_{j+1}$. Since $\angle A_{j-1}A_jA_{j+1}>\pi$, such sub-polygon must be non-convex and this leads to a contradiction.

\textbf{Claim 3}\; If $J_c$ divides $P$ into convex polygons, then $\chi(M_d\setminus J)\ne 0$ if and only if $J_c=J$.

Since $J_c$ divides $P$ into convex polygons, Claim 2 then implies that $J_c$ is the minimal set of $NC_c[J]$.  
So, Proposition \ref{pro:2} (i.e., Theorem \ref{th:2} (2a)) deduces Claim 3.

Now we divide the proof of Theorem \ref{th:main1} (A) into several cases.

\noindent {\bf  Case 1}. $J_c$ divides $P$ into convex polygons. This is equivalent to $\angle A_2A_1A_n<\pi$ and $\angle A_{j+1}A_jA_{j-1}>\pi$ for any $j\ne 1,2,n$.

Since $\chi(M_d\setminus J)\ne 0$, Claim 3 implies that $J_c=J$. That is, $\angle A_{j+1}A_jA_{j-1}>\pi$ for any $j\ne 1,2,n$. Thus, $(n-3)\pi+\angle A_3A_2A_1+\angle A_2A_1A_n+\angle A_1A_nA_{n-1}<\sum_{j=1}^n \angle A_{j+1}A_jA_{j-1}=(n-2)\pi$. We immediately get $\angle A_2A_1A_n<\pi$. So, $P$ belongs to Class \ref{ex:class2}.

\noindent {\bf  Case 2}. $J_c$ doesn't divide $P$ into convex polygons.

In this case, $J_c\ne J$ and $\angle A_2A_1A_n>\pi$.

\noindent \;\, {\sl\textbf{Case 2.1}}. $J_c=\varnothing$, i.e., $A_{j+1}A_jA_{j-1}<\pi$ for any $j\ne 1,2,n$.

In this case, combining with the fact $J\subset M_d$, we further have $\angle A_{j+1}A_jA_{j-1}<\pi$ for any $j\ne 1$. If $\angle A_2A_1A_n<\pi$, then $P$ is a convex polygon. Thus, Corollary \ref{cor:convex=0} deduces that $\chi(M_d\setminus J)= 0$ unless $P$ is a triangle. Next we assume that $\angle A_2A_1A_n>\pi$.

\; {\sl \textbf{Case 2.1.1}}. $NC_{nc}[J]=\{\varnothing\}$, i.e., $\angle A_3A_1A_n<\pi$ and $\angle A_2A_1A_{n-1}<\pi$ (see Fig.~\ref{fig:class1-1}).

In this case, $\chi(M_d\setminus J)=(-1)^{|P|}$.

\; {\sl \textbf{Case 2.1.2}}. $NC_{nc}[J]\setminus\{\varnothing\}\ne\varnothing$.

 Then each maximal $I\in NC_{nc}[J]$ possesses the form $\{A_1A_3,\ldots,A_1A_i\}\cup \{A_1A_j,\ldots,A_1A_{n-1}\}$ with $3\le i$ and $j\le n-1$, and the only non-convex sub-polygon is $A_1A_iA_{i+1}\cdots A_j$ with $\angle A_iA_1A_j>\pi$. Here the first part $\{A_1A_3,\ldots,A_1A_i\}$ or the second part $\{A_1A_j,\ldots,A_1A_{n-1}\}$ may be empty but cannot be both empty, and if the two parts are both nonempty then $i<j$. So, each maximal $I\in NC_{nc}[J]$ contains the diagonal $A_1A_3$ or the diagonal $A_1A_{n-1}$. Let $I_1,\ldots,I_m\in NC_{nc}[J]$ be all the maximal sets. If $\cap_{i=1}^mI_i\ne \varnothing$, then Corollary  \ref{cor:intersection-maximal} implies $\chi(M_d\setminus J)=0$.
 So, $\cap_{i=1}^mI_i= \varnothing$, i.e., there exists $i,j\in \{1,\ldots,m\}$ such that $A_1A_3\not\in I_i$ and $A_1A_{n-1}\not\in I_j$, and clearly, such $i$ and $j$ are unique.  Without loss of generality, we may assume that $A_1A_3\not\in I_m$ and $A_1A_{n-1}\not\in I_1$. Then $A_1A_3\in I_1$ and $A_1A_{n-1}\in I_m$.

If $m>2$, then for any $2\le i\le m-1$, $\{A_1A_3,A_1A_{n-1}\}\subset I_i$. Note that $\xi(I_{i_1}\cap\cdots\cap I_{i_k})=1$ $\Leftrightarrow$ $\xi(I_{i_1}\cap\cdots\cap I_{i_k})\ne 0$
 $\Leftrightarrow$ $\{I_1,I_m\}\subset \{I_{i_1},\ldots, I_{i_k}\}$. It follows from Proposition \ref{pro:xi-eta} that
\begin{align*}
\chi(M_d\setminus J)&=\sum_{k=1}^m (-1)^{k-1}\sum_{1\le i_1<\cdots<i_k\le m}\xi(I_{i_1}\cap\cdots\cap I_{i_k})
=\sum_{k=2}^{m} (-1)^{k-1}\sum_{1=i_1<\cdots<i_k= m}1
\\&=\sum_{k=2}^{m} (-1)^{k-1}{m-2\choose k-2}=-(-1+1)^{m-2}=0,
\end{align*}
which is a contradiction.

Thus, $m=2$, and we can assume  $I_1=\{A_1A_3,\ldots,A_1A_i\}$ and $I_2=\{A_1A_j,\ldots,A_1A_{n-1}\}$. Obviously, $\angle A_3A_1A_{n-1}<\pi$, $\angle A_3A_1A_n>\pi$ and $\angle A_2A_1A_{n-1}>\pi$ (see Fig.~\ref{fig:class1-2}). In this case, $\chi(M_d\setminus J)=(-1)^{|P|}(\xi(I_1)+\xi(I_2)-\xi(I_1\cap I_2))=(-1)^{|P|+1}$.

\noindent \;\, {\sl\textbf{Case 2.2}}. $J_c\ne\varnothing$.

Assume $J_c=\{A_1A_{i_1},\ldots,A_1A_{i_k}\}$, where $3\le i_1<\cdots<i_k\le n-1$. For simplicity, we set $i_0=2$ and $i_{k+1}=n-1$. Note that $\sum_{s=0}^{k}\angle A_{i_s}A_1A_{i_{s+1}}=\angle A_2A_1A_n<2\pi$. So, there is at most one $s\in \{0,1,\ldots,k\}$ such that $\angle A_{i_s}A_1A_{i_{s+1}}>\pi$. If for any $s\in \{0,1,\ldots,k\}$, $\angle A_{i_s}A_1A_{i_{s+1}}<\pi$, then $J_c$ divides $P$ into convex polygons, which contradicts to the assumption of Case 2.

Hence, there exists a unique $t\in \{0,1,\ldots,k\}$ such that $\angle A_{i_t}A_1A_{i_{t+1}}>\pi$ (see Fig.~\ref{fig:glue}). Now we shall prove that for any $j\in \{3,\ldots, i_t\}\cup\{i_{t+1},\ldots,n-2\}$, $\angle A_{j+1}A_jA_{j-1}>\pi$ and thus $J_c=\{A_1A_j:j=3,\ldots, i_t,i_{t+1},\ldots,n-2\}$. If not, then there exists $j_0\in \{3,\ldots, i_t\}\cup\{i_{t+1},\ldots,n-2\}$ such that $\angle A_{j_0+1}A_{j_0}A_{j_0-1}<\pi$, i.e., $A_1A_{j_0}\not\in J_c$. Thus, for any minimal set $I\in NC_c[J]$, $A_1A_{j_0}\not\in I$. Suppose $J_1,\ldots,J_m\in NC_c[J]$ are all the minimal sets. Then $\cup_{i=1}^mJ_i\ne J$ and thus  Corollary \ref{cor:union-minimal} implies that $\chi(M_d\setminus J)=0$, which is a contradiction.

\begin{figure}[H]
\centering
\begin{tikzpicture}[scale=1.5]
\draw (0,0) to (-0.5,-1);
\draw (0,0) to (0.5,-1);
\draw (-0.5,-0.5) to (-0.8,-0.4);
\draw (0.5,-0.5) to (0.8,-0.4);
\draw (-1.5,-0.5) to  (-0.8,-0.4);
\draw (1.5,-0.5) to (0.8,-0.4);
\draw (-1.5,-0.5) to (-0.9,1.2);
\draw (1.5,-0.5) to ( 0.9,1.2);
\draw (-0.9,1.2) to (0,1.2);
\draw ( 0.9,1.2) to (0,1.2);
\draw (-0.4,-0.64) to (-0.5,-0.5);
\draw (0.4,-0.64) to (0.5,-0.5);
\draw (-0.4,-0.64) to (-0.5,-1);
\draw (0.4,-0.64) to (0.5,-1);
\draw[dashed] (0,0) to (-0.8,-0.4);
\draw[dashed] (0,0) to (0.8,-0.4);
\node (A1) at (0,0.2) {$A_1$};
\node (An) at (-0.4,-1.15) {\small $A_n$};
\node (A2) at (0.4,-1.15) {\small $A_2$};
\node (Ait) at (0.85,-0.2) {\small $A_{i_t}$};
\node (Ait+1) at (-0.85,-0.2) {\small $A_{i_{t+1}}$};
\end{tikzpicture}
\caption{\label{fig:glue}Illustration for the proof of Theorem \ref{th:main1} (A). }
\end{figure}
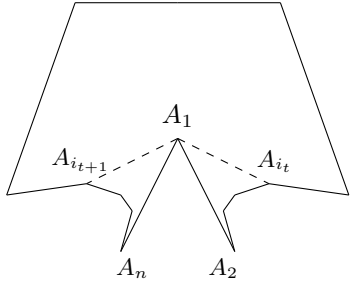

Then Proposition \ref{pro:J'J} implies that $\chi(M_d\setminus J)=\chi(M_d(A_1 A_{i_t}\cdots A_{i_{t+1}})\setminus (J\setminus J_c))$. Note that the sub-polygon $P':=A_1 A_{i_t}\cdots A_{i_{t+1}}$  (see Fig.~\ref{fig:glue}) fulfils the assumption of Case 2.1, i.e., $\angle A_{i_t}A_1A_{i_{t+1}}>\pi$ (fulfils the assumption of Case 2) and $A_{j+1}A_jA_{j-1}<\pi$ for any $j\ne 1,i_t,i_{t+1}$
 (further fulfils the assumption of Case 2.1). And note that the sub-polygons $A_1A_2\cdots A_{i_t}$ (if $i_t\ne 1,2$) and $A_1 A_{i_{t+1}}\cdots A_n$ (if $i_{t+1}\ne n,1$) satisfy the assumption of Case 1  (see Fig.~\ref{fig:glue}). In consequence, $P$ belongs to Class \ref{ex:class6}.
~\\

(B) If $P$ is a convex polygon, then $M_e=\varnothing$,  $\chi(M_e\setminus J)=\chi(\varnothing)=1$, and thus the statement obviously holds.

Next we focus on the non-convex case. Then the boundary of the convex hull of $P$ forms a convex polygon which is denoted by $Pconv(P)$, and each edge of the convex polygon is either an edge of $P$ or an epigonal of $P$. Since $\chi(M_e\setminus J)\ne 0$, all the edges of the convex polygon $Pconv(P)$ which are the epigonals of $P$ must belong to $J=\{A_1A_j:j\ne 1,2,n\}$. Thus, $A_1$ is a vertex of $Pconv(P)$  and the number of such epigonals which are edges of $Pconv(P)$ is at most two.

We may assume without loss of generality that there are exact two epigonals which are edges of $Pconv(P)$ with the common vertex $A_1$, denoted by $A_1A_{j_1}$ and $A_1A_{j_2}$. Then there are two polygons between $P$ and $Pconv(P)$, denoted them by $P^1$ and $P^2$ which respectively possesses the edges $A_1A_{j_1}$ and $A_1A_{j_2}$.

So, by Lemma \ref{lem:e}, $ \chi(M_e(P))=\chi(M_d(P^1\setminus J))\chi(M_d(P^2\setminus J))$, and thus we have $\chi(M_d(P^1\setminus J))\ne 0$ and $\chi(M_d(P^2\setminus J))\ne 0$. Theorem \ref{th:main1} (A) shows that $P^1$ and $P^2$ must belong to Class \ref{ex:class1} or Class \ref{ex:class2} or Class \ref{ex:class6}. Note that the sum of the angles at $A_1$ of $P^1$ and the angles at $A_1$ of $P^2$ is less that $\angle A_{j_1}A_1A_{j_2}<\pi$. So, $P^1$ and $P^2$ must belong to Class \ref{ex:class2}, and thus $P$ must belong to Class \ref{ex:class3}.
~\\

(C)  If $A_2A_n\not\in M_d$, then $P$ is non-convex and $M_d\setminus\{ A_2A_n\}=M_d$. Hence, $\chi(M_d\setminus\{ A_2A_n\})=\chi(M_d)=0$.

Next we assume that $A_2A_n\in M_d$. Then $\angle A_2A_1A_n< \pi$.

If $P$ is a convex polygon, then Corollary \ref{cor:convex=0} deduces that $\chi(M_d\setminus\{ A_2A_n\})=0$.

If $P$ is a non-convex polygon and $\chi(M_d\setminus\{ A_2A_n\})\ne 0$, then Proposition \ref{pro:1} implies that $A_2A_n$ divides $P$ into convex polygons. Hence, the sub-polygon $A_2A_3\cdots A_n$ is convex, and as a consequence, $P$ belongs to Class \ref{ex:class4}.
~\\

(D) If $P$ is a convex polygon, then $\chi(M_e\setminus\{ A_2A_n\})\ne0$. So, we only concentrate on the non-convex case.

If $\chi(M_e\setminus\{A_2A_n\})\ne0$, then $A_2A_n$ is the unique edge of $Pconv(P)$ which is not an edge of  $P$. Such polygon must belong to Class \ref{ex:class2} or Class \ref{ex:class5}.

For Class \ref{ex:class2}, $\chi(M_e\setminus\{ A_2A_n\})=\chi(M_d(A_2A_3\cdots A_n))=(-1)^{n}$. For Class \ref{ex:class5}, $\chi(M_e\setminus\{ A_2A_n \})=\chi(\varnothing)=1$.

We have completed the proof of Theorem \ref{th:main1}.

\section{Proof of Theorem \ref{d_k}}
\label{sec:furtherresults}

First we give some elementary facts for $a$-diagonals of a polygon $P$.

\begin{remark}
\begin{enumerate}[(1)]
\item Each diagonal is an $1$-diagonal.

\item Each $ab$-diagonal is both $a$-diagonal and $b$-diagonal.

\item A polygon $P$ is called an $a$-polygon if it has $a(n+1)+2$ vertices for some $n\in\mathbb{N}$. Every $a$-diagonal of $P$ divides $P$ into two $a$-polygons.

\item If $P$ is convex and $|P|=a(n+1)+2$, then there exists $n$ $a$-diagonals with one common endpoint.
\end{enumerate}
\end{remark}

Let $P=P_{a(n+1)+2}$ be a convex polygon with $a(n+1)+2$ vertices. Let $d_1(n,a)$ be the number of $a$-diagonals, $d_2(n,a)$ be the number of non-crossing pairs of $a$-diagonals, and, in general, $d_i(n,a)$ be the number of sets of $i$ $a$-diagonals of the polygon which are pairwise non-crossing.

\begin{pro}\label{pro:d_k}
$d_k(n,a)=\frac{1}{k+1}{n\choose k}{a(n+1)+k+1\choose k}$.
\end{pro}
\begin{proof}
Corollary 6  \cite{ZWZ} (or Theorem 4 \cite{Michael}) gives that the number of different ways of cutting $P_{a(n+1)+2}$ into sub-polygons $P_{ai_1+2},P_{ai_2+2},\ldots,P_{ai_{k+1}+2}$ by diagonals is always $\frac{1}{k+1}{a(n+1)+k+1\choose k}$, where $(i_1,\ldots,i_{k+1})$ is a given ordered array of positive integers.

 Note that $\sum_{j=1}^{k+1}|P_{ai_j+2}|=|P_{a(n+1)+2}|+2k$, i.e., $a(i_1+\cdots+i_{k+1})+2(k+1)=a(n+1)+2+2k$, and this is equivalent to $i_1+i_2+\cdots+i_{k+1}=n+1$. Since the number of positive integer solutions of $i_1+i_2+\cdots+i_{k+1}=n+1$ is ${n\choose k}$, we have $d_k(n,a)=\frac{1}{k+1}{a(n+1)+k+1\choose k}{n\choose k}$. 
\end{proof}

It should be noted that Proposition \ref{pro:d_k} is nothing but Corollary 2 \cite{Przy}. Here we show a new and simple  proof of such result above.

\begin{pro}\label{pro:generalLee}
For any $a\in \mathbb{N}^+$, we have
\[\sum_{k=1}^{n}(-1)^{k-1}d_k(n,a)= 1+(-1)^{n+1}\frac{{a(n+1)\choose n}}{n+1}=1+(-1)^{n+1}d_{n}(n,a-1).\]
\end{pro}

Proposition \ref{pro:generalLee} can be proved by modifying the ideas in \cite{ZWZ}.

\begin{proof}
\begin{align*}
\sum_{k=1}^{n}(-1)^{k-1}d_k(n,a)&=\sum_{k=1}^{n}(-1)^{k-1}\frac{{n\choose k}}{k+1}{a(n+1)+k+1\choose k}
\\&=\sum_{k=1}^{n}(-1)^{k-1}\frac{{n+1\choose k+1}}{n+1}\res
\left(\frac{(1+u)^{a(n+1)+k+1}}{u^{k+1}}\right)
\\&=\frac{1}{n+1}\res
\left((1+u)^{a(n+1)}\sum_{k=1}^{n}(-1)^{k-1}{n+1\choose k+1}\frac{(1+u)^{k+1}}{u^{k+1}}\right)
\\&=\frac{1}{n+1}\res
\left((1+u)^{a(n+1)}\left[(1-\frac{u+1}{u})^{n+1}-(1-(n+1)\frac{u+1}{u})\right]\right)
\\&=\frac{1}{n+1}\res
\left((1+u)^{a(n+1)}\left[(-\frac{1}{u})^{n+1}-1+(n+1)\frac{u+1}{u}\right]\right)
\\&=\frac{1}{n+1}\left\{(-1)^{n+1} \res
\left(\frac{(1+u)^{a(n+1)n}}{u^{n+1}}\right)+(n+1)\res \left(\frac{(1+u)^{a(n+1)+1}}{u}\right)\right\}
\\&=\frac{1}{n+1}\left\{(-1)^{n+1} {a(n+1)\choose n}+(n+1)\cdot 1\right\}
\\&=1+(-1)^{n+1}\frac{{a(n+1)\choose n}}{n+1}
=1+(-1)^{n+1}d_n(n,a-1),
\end{align*}
where $\res(f(u))$ is the residue of the function $f(u)$ at $u=0$.
\end{proof}

\begin{remark}
Since $d_k(n,1)=d_k(n+1)$ and $d_n(n,0)=0$, Lee's theorem (i.e., Theorem \ref{th:main} (A))
\[{d_1} - {d_2} + {d_3} -  \cdots  + {( - 1)^n}{d_{n - 1}} = 1 + {( - 1)^n}\]
is clearly a special case of Proposition \ref{pro:generalLee} for $a=1$.
\end{remark}

\begin{pro}\label{pro:d-k}
Given $k,n\in \mathbb{N}^+$, we have
$$d_k(n,a)=\frac{a(n+1)+2}{2k}\sum_{i_1+i_2=n-1}\sum_{j_1+j_2=k-1}d_{j_1}(i_1,a)d_{j_2}(i_2,a).$$
\end{pro}
\begin{remark}
This result is a generalization of the identity of Catalan's number. The proof is standard and hence we omit it.
\end{remark}

By Proposition \ref{pro:d_k} and Proposition \ref{pro:d-k}, we have the following combinatorial identity. Here, we present another proof by residue theorem and PDE method.
\begin{pro}\label{pro:calatanequality}
Given $n,i\in \mathbb{N}^+$, we have
\begin{equation}\label{eq:calatanequality}
\frac{{a(n + 1) + (i + 1)\choose i}{n\choose i}}{{i + 1}} = \frac{{a(n + 1) + 2}}{{2i}}\sum\limits_{
\scriptsize{\begin{array}{c} {n_1} + {n_2} = n - 1,\\ {i_1} + {i_2} = i - 1\end{array}}} { {\frac{{{a({n_1} + 1) + ({i_1} + 1)\choose {i_1}}{a({n_2} + 1) + ({i_2} + 1)\choose {i_2}}{{n_1}\choose {i_1}}{{n_2}\choose {i_2}}}}{{({i_1} + 1)({i_2} + 1)}}} }.
\end{equation}
\end{pro}


\begin{proof}
Let $a_{n,i}=\frac{1}{i+1}{a(n+1)+(i+1)\choose i}{n\choose i}$ and
\begin{equation}\label{eq:F(x,y)}
F(x,y)=\sum_{n,i\geq 0}a_{n,i}x^ny^i.
\end{equation}

Then $xyF^2(x,y)=\sum_{n,i\geq 0}b_{n,i}x^ny^i$. \eqref{eq:calatanequality} is equivalent to $a_{n,i}=\frac{a(n+1)+2}{2i}b_{n,i}$, $n,i>0$, which can be written as $2i a_{n,i}=(a(n+1)+2)b_{n,i},~n,i\geq0$. Note that $i=0$ or $i>n$ implies $b_{n,i}=0$, then it follows from $\sum_{n,i\geq0}(a(n+1)+2)b_{n,i}x^ny^i=a(x^2yF^2)_x+2xy F^2$ and $\sum_{n,i\geq0}2i a_{n,i}x^ny^i=2yF_y$ that  \eqref{eq:calatanequality} is equivalent to $a(x^2yF^2)_x+2xyF^2=2yF_y$. This can be simplified as  $a(2xyF^2+2x^2yFF_x)+2xy F^2=2yF_y$, which can be further written as
\begin{equation}\label{eq:partial-F}
(1+a)xF^2+ax^2FF_x=F_y.
\end{equation}

Next we use the method of characteristics to solve \eqref{eq:partial-F}.

Let $x=x(y)$ solve $\frac{dx}{dy}=-ax^2F(x,y)$. Then $F=F(x(y),y)$ solves $\frac{dF}{dy}=(1+a)F^2x$, and we have $\frac{d (xF)}{dy}=F^2x^2$.

Thus we have $xF=\frac{1}{c_1-y}$ for some $c_1\in \mathbb{R}$, and then $\frac{dx}{dy}=\frac{-ax}{c_1-y}$, $x=c_2(c_1-y)^a$ for some $c_2\in \mathbb{R}$. Therefore, $F=\frac{1}{c_2(c_1-y)^{a+1}}$. Taking the initial datum $y=0$ in \eqref{eq:F(x,y)}, we have $F(x,0)=\sum_{n\geq0}a_{n,0}x^n=\sum_{n\ge 0} x^n=\frac{1}{1-x}$, and thus $\frac{1}{1-c_2c_1^a}=\frac{1}{c_2c_1^{a+1}}$. So, $c_2c_1^a(c_1+1)=1$, $x=\frac{(c_1-y)^a}{c_1^a(c_1+1)}$, $F=\frac{c_1^a(c_1+1)}{(c_1-y)^{a+1}}$.

Let $t=\frac{y}{c_1}$. Then $x=\frac{t(1-t)^a}{y+t}$, $F=\frac{y+t}{(1-t)^{a+1}y}$. Let $t=xv$, we have
\begin{equation}\label{eq:Fxv}
y=v((1-xv)^a-x),\ \ \ F=\frac{1}{(1-xv)((1-xv)^a-x)}.
\end{equation}
According to the implicit function theorem, $v$ and then $F$ must be an analytic function of $(x,y)$ for sufficiently small $|x|$ and $|y|$, thus  \eqref{eq:Fxv} gives a solution of \eqref{eq:partial-F}. Next we prove that the $F$ satisfying \eqref{eq:Fxv} must satisfy \eqref{eq:F(x,y)}.

Let $F=\sum_{i\geq0}f_i(x)y^i$. Then by the residue theorem, we have
\begin{align*}
f_i(x)&=\res(Fy^{-i-1}dy)=\res\left(\frac{(1-xv)^ax-axv(1-xv)^{a-1}}{(1-xv)((1-xv)^a-x)^{i+2}v^{i+1}}dv \right)
\\&=\res\left(\frac{dv}{(1-xv)^{1+a(i+1)}v^{i+1}}\left(1-\frac{x}{(1-xv)^a}\right)^{-i-1}-
\frac{axdv}{(1-xv)^{2+a(i+1)}v^i}\left(1-\frac{x}{(1-xv)^a}\right)^{-i-2} \right)
\\&=\res\left(\frac{dv}{(1-xv)^{1+a(i+1)}v^{i+1}}\sum_{n=i}^{+\infty}\frac{{n\choose i}x^{n-i}}{(1-xv)^{a(n-i)}}-
\frac{axdv}{(1-xv)^{2+a(i+1)}v^i}\sum_{n=i+1}^{+\infty}\frac{{n\choose i+1}x^{n-i-1}}{(1-xv)^{a(n-i-1)}} \right)
\\&=\sum_{n=i}^{+\infty}\res\left(\frac{dv x^{n-i}{n\choose i}}{(1-xv)^{1+a(n+1)}v^{i+1}} \right)-\sum_{n=i+1}^{+\infty}\res\left(\frac{ax^{n-i}dv {n\choose i+1}}{(1-xv)^{2+an}v^i}\right)
\\&=\sum_{n=i}^{+\infty}x^n{n\choose i}{1+a(n+1)+i-1\choose i}-\sum_{n=i+1}^{+\infty}ax^{n-1}{n\choose i+1}{2+an+i-2\choose i}
\\&=\sum_{n=i}^{+\infty}x^n\left({n\choose i}{a(n+1)+i\choose i}-a{n+1\choose i+1}{a(n+1)+i\choose i}\right)
\\&=\sum_{n=i}^{+\infty}x^n {n\choose i}{a(n+1)+i\choose i}\left(1-a\frac{n+1}{i+1}\frac{i}{a(n+1)+1}\right)
\\&=\sum_{n=i}^{+\infty}x^n {n\choose i}{a(n+1)+i\choose i}\frac{a(n+1)+i+1}{(i+1)(a(n+1)+1)}
\\&=\sum_{n=i}^{+\infty}x^n {n\choose i}{a(n+1)+i+1\choose i}\frac{1}{i+1}
=\sum_{n\geq 0}x^na_{n,i}.
\end{align*}
Therefore, \eqref{eq:F(x,y)} holds, and then  \eqref{eq:calatanequality} holds.
\end{proof}

{\bf Acknowledgement}\;\; The third-named author thanks Zipei Nie for interesting discussions.

\section*{Appendix}

\begin{proof}[Proof of Proposition \ref{pro:d1>=1}]
\begin{figure}[H]
\centering
\begin{tikzpicture}[scale=0.8]
\draw (0,0) to (4,2.4);
\draw (5,3) to (4,2.4);
\draw (0,0) to (8,-2);
\draw[red] (0,0) to (4.6,1.4);
\draw[densely dotted] (4,2.4) -- (6.4,-1.6);
\draw[densely dotted]  (5,3)--(8,-2);
\draw[dashed]  (5,3)--(7,2.8);
\draw[dashed]  (8,-2)--(8.8,-1.6);
\draw[dashed]  (9.8,-0.3)--(8.8,-1.6);
\draw[dashed]  (4.6,1.4)--(6.7,1.7);
\draw[dashed]  (7,2.8)--(6.7,1.7);
\draw[dashed]  (4.6,1.4)--(6.9,1.2);
\draw[dashed]  (9.8,-0.3)--(6.9,1.2);
\node (A) at (-0.2,-0.2) {\small $A$};
\node (B) at (5,3.2) {\small$B$};
\node (C) at (8,-2.2) {\small$C$};
\node (B1) at (4,2.6) {\small$B'$};
\node (C1) at (6.4,-1.8) {\small$C'$};
\node (D) at (4.6,1.2) {\small$D$};
\end{tikzpicture}

\caption{\label{fig:exist-diagonal}A figure used in Proposition \ref{pro:d1>=1}.}
\end{figure}
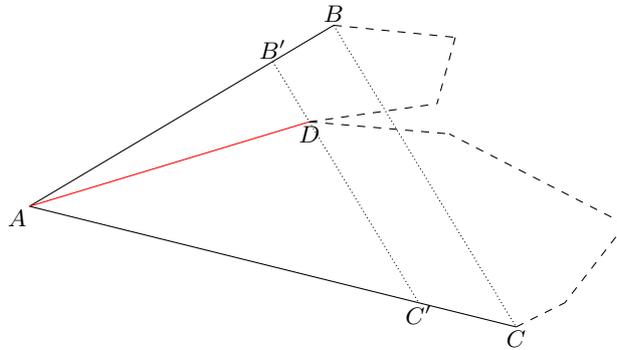
We set $A_0=A_n$ and $A_{n+1}=A_1$. Since $\sum_{i=1}^n \angle A_{i-1}A_iA_{i+1}=(n-2)\pi$, there exists $i_0\in\{1,2,\ldots,n\}$ such that $\angle A_{i_0-1}A_{i_0}A_{i_0+1}\le \frac{(n-2)\pi}{n}<\pi$. For simplicity, we let $B=A_{i_0-1}$, $A=A_{i_0}$ and $C=A_{i_0+1}$. Then $\angle BAC<\pi$, and the segment $BC$ is not an edge of $P_n$ (otherwise, $P_n=\triangle ABC$ and this contradicts with $n\ge 4$).

If $BC$ is a diagonal, then there is nothing need to show.

If $BC$ is not a diagonal, then there is a vertex $D$ inside $\triangle ABC$ with greatest distance to $BC$
(see Fig.~\ref{fig:exist-diagonal}). 
Accordingly, $AD$ lies in the polygonal region, and thus $AD$ must be a diagonal.
\end{proof}

\begin{proof}[Proof of Proposition \ref{pro:triangulation}]
Note that $J$ divides $P$ into some polygons, which can be denoted by $P_1,\ldots,P_k$. Clearly, each diagonal of a sub-polygon $P_i$ is a diagonal of $P$, where $i\in \{1,\ldots,k\}$. Proposition \ref{pro:d1>=1} yields that if there exists $P_i$ with $|P_i|\ge 4$, then $P_i$ has a diagonal, and we can add the diagonal  to $J$. Repeat the process until every sub-polygon is a triangle. At this time, we obtain $J'$, which provides a triangulation of $P_n$. Obviously, $\#J'=n-3$ and $J'\supset J$.
\end{proof}

\begin{proof}[Proof of Proposition \ref{pro:chi-induction}]
By the definition of $\nu_i(A)$, for $i\ge 1$, there holds
\begin{align*}
\nu_i(A)&=\sum_{\# B=i,~B\in NC[A]}1
\\&=\sum_{\# B=i,~v\in B\in NC[A]}1+\sum_{\# B=i,~v\not\in B\in NC[A]}1
\\&=\sum_{\# B'=i-1,~B'\in NC[A_v]}1+\sum_{\# B=i,~ B\in NC[A\setminus \{v\}]}1
\\&=\nu_{i-1}(A_v)+\nu_i(A\setminus \{v\}).
\end{align*}
Then the proof of $\chi(A)=\chi(A\setminus \{v\})-\chi(A_v)$ is immediately completed by taking alternating sum. 
\end{proof}

\subsection*{
Precise construction processes of the polygons used in the proof of Proposition \ref{pro:every} }

We use complex coordinate. Denote $\omega_k=(1+(-1)^k\sqrt{3}\mathbf{i})/{2}$ and set the points $B_k,C_k,D_k$ such that $\forall\ k\in\mathbb{Z}, $ $$B_{2k}-B_{2k-1}=1,\ B_{2k+1}-B_{2k}=e^{\frac\pi6\mathbf{i}},\ C_k-B_k=\omega_k(B_{k+1}-B_{k}),\ D_k-B_k=\omega_k(B_{k+2}-B_{k}).$$
Then we can choose
$$A_{3k-1}^0=B_{2k},\ A_{3k}^0=C_{2k+1},\ A_{3k+1}^0=D_{2k+1},\ A_{3k-3}^1=C_{2k},\ A_{3k-2}^1=D_{2k},\ A_{3k-1}^1=B_{2k+1}, \forall\ k\in\mathbb{Z}.$$
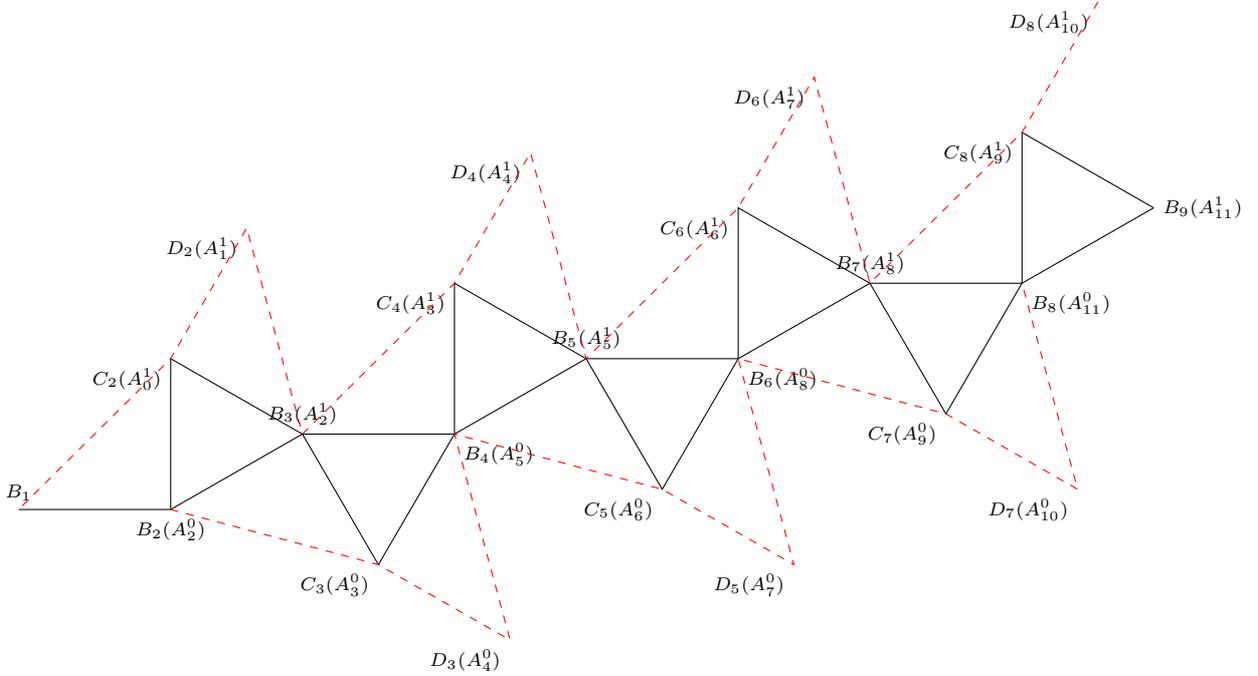
\begin{figure}[H]
\centering
\begin{tikzpicture}[scale=2]
\coordinate (B1) at (0,0);
\coordinate (C1) at (0.5,-0.866);
\coordinate (D1) at (1.366,-1.366);

\coordinate (B2) at (1,0);
\coordinate (C2) at (1,1);
\coordinate (D2) at (1+0.5,1.866);
\coordinate (B3) at (1+0.866,0.5);
\coordinate (C3) at (1+0.866+0.5,0.5-0.866);
\coordinate (D3) at (1+0.866+1.366,0.5-1.366);

\coordinate (B4) at (1+0.866+1,0.5);
\coordinate (C4) at (1+0.866+1,0.5+1);
\coordinate (D4) at (1+0.866+1+0.5,0.5+1.866);

\coordinate (B5) at (1+0.866+1+0.866,0.5+0.5);
\coordinate (C5) at (1+0.866+1+0.866+0.5,0.5+0.5-0.866);
\coordinate (D5) at (1+0.866+1+0.866+1.366,0.5+0.5-1.366);

\coordinate (B6) at (1+0.866+1+0.866+1,0.5+0.5);
\coordinate (C6) at (1+0.866+1+0.866+1,0.5+0.5+1);
\coordinate (D6) at (1+0.866+1+0.866+1+0.5,0.5+0.5+1.866);

\coordinate (B7) at (1+0.866+1+0.866+1+0.866,0.5+0.5+0.5);
\coordinate (C7) at (1+0.866+1+0.866+1+0.866+0.5,0.5+0.5+0.5-0.866);
\coordinate (D7) at (1+0.866+1+0.866+1+0.866+1.366,0.5+0.5+0.5-1.366);

\coordinate (B8) at (1+0.866+1+0.866+1+0.866+1,0.5+0.5+0.5);
\coordinate (C8) at (1+0.866+1+0.866+1+0.866+1,0.5+0.5+0.5+1);
\coordinate (D8) at (1+0.866+1+0.866+1+0.866+1+0.5,0.5+0.5+0.5+1.866);

\coordinate (B9) at (1+0.866+1+0.866+1+0.866+1+0.866,0.5+0.5+0.5+0.5);
\coordinate (C) at (2*\pgfmathresult, 1);

\draw (B1)node [above]{\scriptsize $B_1$}
      --(B2)node [below]{\scriptsize$B_2(A_2^0)$}
      --(B3)node [above]{\scriptsize$B_3(A_2^1)$}
      --(B4)node [below right]{\scriptsize$B_4(A_5^0)$}
      --(B5)node [above]{\scriptsize$B_5(A_5^1)$}
      --(B6)node [below right]{\scriptsize$B_6(A_8^0)$}
      --(B7)node [above]{\scriptsize$B_7(A_8^1)$}
      --(B8)node [below right]{\scriptsize$B_8(A_{11}^0)$}
      --(B9)node [ right]{\scriptsize$B_9(A_{11}^1)$}
      ;
\draw  (B3) -- (C2) node [below left]{\scriptsize$C_2(A_0^1)$} -- (B2)(D2) node [below left]{\scriptsize$D_2(A_1^1)$};
\draw  (B4) --(C3) node [below left]{\scriptsize$C_3(A_3^0)$} -- (B3)  (D3) node [below left]{\scriptsize$D_3(A_4^0)$};
\draw  (B5) --(C4) node [below left]{\scriptsize$C_4(A_3^1)$} -- (B4)  (D4) node [below left]{\scriptsize$D_4(A_4^1)$};
\draw  (B6) --(C5) node [below left]{\scriptsize$C_5(A_6^0)$} -- (B5)  (D5) node [below left]{\scriptsize$D_5(A_7^0)$};
\draw  (B7) --(C6) node [below left]{\scriptsize$C_6(A_6^1)$} -- (B6)  (D6) node [below left]{\scriptsize$D_6(A_7^1)$};     
\draw  (B8) --(C7) node [below left]{\scriptsize$C_7(A_9^0)$} -- (B7)  (D7) node [below left]{\scriptsize$D_7(A_{10}^0)$};
\draw  (B9) --(C8) node [below left]{\scriptsize$C_8(A_9^1)$} -- (B8)  (D8) node [below left]{\scriptsize$D_8(A_{10}^1)$};
\draw[dashed,red] (B2)--(C3)--(D3)--(B4)--(C5) -- (D5)--(B6) -- (C7)--(D7)--(B8);

\draw[dashed,red] (D8)--(C8)--(B7)--(D6)--(C6)--(B5)--(D4)--(C4)--(B3)--(D2)--(C2)--(B1);
\end{tikzpicture}
\caption{\label{fig:Appendix} A figure used in Appendix.  }
\end{figure}

And next we take (for the case of $l>1$)
$$A_1^2=B_1,\ A_{k}^2=A_{k}^0,\ \forall\ 2\leq k\leq (3l+1)/2,\ k\in\mathbb{Z},\  A_{3l-k}^2=A_{k}^1,\ \forall\ 0\leq k\leq (3l-2)/2,\ k\in\mathbb{Z}.$$
Then the polygon with linearly ordered counter-clockwise vertices $A_1^2,A_2^2,\ldots,A_{3l}^2$ is our desired polygon $P$. Although many vertices are collinear, we can avoid this by small perturbations, i.e., replace the vertices $A_1^2,A_2^2,\ldots,A_{3l}^2$ by the points $A_1,A_2,\ldots,A_{3l}$ which are in general position and  $|A_k-A_k^2|<1/100$, $k=1,\ldots,3l$.
\end{document}